\documentclass[12pt]{article}
\newtheorem{theorem}{\bf Theorem}[section]

\usepackage{marginnote}
\usepackage{amssymb,epstopdf}
\usepackage{amsfonts}
\usepackage{amsmath}
\usepackage{mathtools} 
\usepackage{color}
\usepackage{graphicx, float}
\usepackage{tikz}
\usepackage{cite}
\usepackage{epsfig,mathrsfs}
\usepackage[tight,TABTOPCAP]{subfigure}
\usepackage{color}
\usepackage{fancyhdr}
\newcommand{\proofend}{\hfill $\Box$ \vspace{2mm}}

\newtheorem{Prop}{Proposition}

\newtheorem{lemma}{Lemma}

\newtheorem{remark}{Remark}

\newcommand{\R}{\mathbb{R}}

\newcommand{\bv}{\boldsymbol{v}}

\newcommand{\be}{\boldsymbol{e}}

\newcommand{\bI}{\boldsymbol{I}}
\newcommand{\mG}{\mathcal{G}}
\newcommand{\mH}{\boldsymbol{\mathcal{H}}}

\newcommand{\bk}{{\boldsymbol{k}}}
\newcommand{\beps}{{\boldsymbol{\varepsilon}}}
\newcommand{\bgam}{{\boldsymbol{\gamma}}}
\newcommand{\bsig}{{\boldsymbol{\sigma}}}

\newcommand{\bmu}{{\boldsymbol{\mu}}}
\newcommand{\brho}{{\boldsymbol{\rho}}}
\newcommand{\bchi}{{\boldsymbol{\chi}}}
\newcommand{\boldeta}{{\boldsymbol{\eta}}}
\newcommand{\balpha}{{\boldsymbol{\alpha}}}
\newcommand{\bzeta}{{\boldsymbol{\zeta}}}
\newcommand{\bx}{\boldsymbol{x}}
\newcommand{\by}{\boldsymbol{y}}
\newcommand{\bg}{\boldsymbol{g}}
\newcommand{\bzero}{\boldsymbol{0}}
\newcommand{\bgamma}{\boldsymbol{\gamma}}
\newcommand{\bnu}{\boldsymbol{\nu}}
\newcommand{\vel}{\langle\tilde{\textrm{\slshape{v}}}\rangle}
\newcommand{\tran}{^{\mbox{\tiny T}}}
\newcommand{\hh}{\hspace*{0.7pt}}
\newcommand{\ww}{w}

\begin{document}

\title{On the dynamic homogenization of periodic media: Willis' approach versus two-scale paradigm}
 \pagenumbering{arabic}
\author{
Shixu Meng\footnote{Institute for Mathematics and its Applications, University of Minnesota, Twin Cities, U.S.
\;(shixumen@umich.edu)}
\quad
and
\quad
Bojan Guzina\footnote{Department of Civil, Environmental \& Geo-Engineering, University of Minnesota, Twin Cities, U.S.
\;(guzin001@umn.edu)}
}
\date{}
\maketitle
%
%
%
%

\begin{abstract}
When considering an effective i.e. homogenized description of waves in periodic media that transcends the usual quasi-static approximation, there are generally two schools of thought: (i) the two-scale approach that is prevalent in mathematics, and (ii) the Willis' homogenization framework that has been gaining popularity in engineering and physical sciences. Notwithstanding a mounting body of literature on the two competing paradigms, a clear understanding of their relationship is still lacking. In this study we deploy an effective impedance of the scalar wave equation as a lens for comparison and establish a low-frequency, long-wavelength (LF-LW) dispersive expansion of the Willis effective model, including terms up to the second order. Despite the intuitive expectation that such obtained effective impedance coincides with its two-scale counterpart, we find that the two descriptions \emph{differ by a modulation factor} which is, up to the second order, expressible as a polynomial in frequency and wavenumber. We track down this inconsistency to the fact that the two-scale expansion is commonly restricted to the \emph{free-wave} solutions and thus fails to account for the body source term which, as it turns out, must also be homogenized -- by the reciprocal of the featured modulation factor. In the analysis, we also (i) reformulate for generality the Willis' effective description in terms of the eigenfunction approach, and (ii) obtain the corresponding modulation factor for \emph{dipole} body sources, which may be relevant to some recent efforts to manipulate waves in metamaterials.
\end{abstract}

\maketitle

\section{Introduction}  
\label{Introduction}
In recent years, periodic composites have been used with remarkable success to manipulate waves toward achieving cloaking, sub-wavelength imaging, and noise control~\cite{MBW06,ZC11,M13} thanks to the underpinning phenomena of frequency-dependent anisotropy and band gaps~\cite{B53}. Commonly the analyses of waves in unbounded periodic media are based on the Floquet-Bloch analysis~\cite{K12} which yields the germane dispersion surfaces, including frequency bands where the free-wave solutions cannot exist. The full understanding of wave interaction with \emph{bounded} periodic domains, however, requires the solution of a relevant boundary value problem~\cite{CGM16}. In situations where the wavelength exceeds the characteristic length scale of medium periodicity~\cite{FG08}, one is compelled to both (i)~gain the physical intuition and (ii)~reduce the computational effort by considering an \emph{effective} i.e. ``macroscopic'' description of the wave motion. Naturally, such an idea raises the fundamental question of the (enriched) governing equation for the mean fields. 

One keen approach to the macroscopic wave description, that has attracted major attention in recent years~\cite{MW07, MBW06, NHA15, NHA16, NSK12, W97, W09, W11}, is the concept of effective constitutive relationships -- proposed by Willis in the early 1980s \cite{W80, W81a, W81b, W83, W84, W85}. In this framework that is often formulated via plane-wave expansion, the non-local effects due to microstructure are encoded in a frequency- and wavenumber-dependent constitutive law that features the coupling terms linking (a)~stress to particle velocity, and (b)~momentum density to strain. Typically, such effective constitutive law is derived via the Green's function approach~\cite{NHA15,W11} that may exhibit instabilities when the frequency-wavenumber pair resides on a dispersion branch~\cite{NSK12}. When considering the space-time formulation, the Willis' model leads to an integro-differential governing equation for the mean fields, whose kernels are given by the inverse Fourier transforms of the effective constitutive parameters. The major appeal of this framework, however, resides in the fact that the Willis' model can be deemed \emph{exact}\cite{NSK12}, since no approximations -- and in particular no asymptotic expansions -- are made in the derivation. In this vein, the Willis' theory carries the potential of capturing the mean wave motion \emph{beyond} the first (i.e. ``acoustic'') dispersion branch, see Fig.~\ref{HomogenizableFigure} for example. 

Within the framework of applied mathematics, on the other hand, the standard approach to extracting effective wave motion at long wavelengths is that of (asymptotic) two-scale homogenization~\cite{B76,PBL78,A92}, where the perturbation parameter signifies the ratio between the unit cell of periodicity and wavelength. By considering the leading-order approximation~\cite{B74,PA06}, one inherently arrives at the quasi-static effective model, where the periodic coefficients in the original field equation are superseded by suitable constants (the so-called effective medium properties). To capture the incipient dispersive effects -- as carried by the acoustic branch, higher-order asymptotic expansions of the effective wave motion were considered e.g. in~\cite{BA93,CF01,AB08,WG15}, resulting in a (constant-coefficient) singular perturbation of the germane field equation.

So far, however, the connection between the Willis' effective model and the two-scale approach to dynamic homogenization is less than clear. For instance in~\cite{NHA16} the authors pursued the long-wavelength, low-frequency (LW-LF) asymptotic expansion of the Willis' model and demonstrated, to the leading order, that such approximation recovers the quasi-static result of two-scale homogenization. This poses the fundamental question: do the two formulations still agree at higher orders of approximation -- which carry the dispersion effects? Indeed we shall show for the first time that the two approximations \emph{differ} at the second order. In particular, we demonstrate that the second-order Willis' and two-scale impedance functions differ by a modulation factor, expressible as a polynomial in the wavenumber-frequency domain. We rigorously link this inconsistency to the fact that the two-scale homogenization is commonly restricted to the free-wave solutions and thus fails to account for the body source term which, as it turns out, must also be homogenized. To begin the analysis, however, we first reformulate the Willis' effective model via the eigenfunction approach which has the benefits of (i)~maintaining the stability across dispersion curves, and (ii)~providing a deeper understanding of the phenomena of crossing dispersion curves and eigenmodes of zero mean that are invisible to the effective model.

Through this investigation, we help establish a rigorous mathematical connection between the two mainstream approaches to dynamic homogenization, and we equip the two-scale approach to handle (monopole and dipole) body sources that may help further manipulate waves in periodic structures~\cite{CW13,NCN17}. The results of this study may also be useful toward extending the applicability of two-scale homogenization beyond the acoustic branch, and tackling the emerging subject of deep sub-wavelength sensing -- where the macroscopic waves are used to interrogate the underpinning microstructure~\cite{RSB15}. 

\section{Preliminaries}

With reference to an orthonormal vector basis~$\be_j$ ($j\!=\!\overline{1,d}$), consider the time-harmonic wave equation
\begin{equation}\label{PDE}
-\omega^2\rho(\bx)u - \nabla\cdot\big(G(\bx)(\nabla u-\boldsymbol{\gamma})\big) ~=~ f \qquad\text{in~~}\mathbb{R}^d, \quad d=1,2,3
\end{equation}
at frequency~$\omega$, where~$G$ and~$\rho$ are $Y$-periodic; 
\[
Y \,=\, \{\bx\!:\,0< \bx\!\cdot\!\be_j < \ell_j; \, j=\overline{1,d}\} 
\]
is the unit cell illustrated in Fig.~\ref{HomogenizableFigure}a), and~$f(\bx)$ (resp.~$\bgamma(\bx)$) denotes the monopole (resp. dipole) source term.  In what follows, $G$ nd~$\rho$ are further assumed to be real-valued $L^\infty(Y)$ functions bounded away from zero. To facilitate the discussion, one may conveniently interpret~(\ref{PDE}) in the context of elasticity and anti-plane shear waves, in which case~$u,G,\rho,f$ and~$\bgamma$ take respectively the meanings of transverse displacement, shear modulus, mass density, body force, and eigenstrain. 

Recalling the plane wave expansion approach~\cite{NHA15, NHA16, NSK12}, consider next the Bloch-wave solutions of the form $u(\bx) = \tilde{u}(\bx) e^ {i\bk\cdot\bx}$, where~$\tilde{u}$ is $Y$-periodic and depends implicitly on~$\bk$ and~$\omega$ -- which are hereon assumed to be fixed. If further the source terms are taken in the form of (i)~plane-wave body force $f(\bx) = \tilde{f} e^ {i\bk\cdot\bx}$ and (ii)~eigenstrain field $\bgam(\bx) = \tilde{\bgam} e^ {i\bk\cdot\bx}$ where $\tilde{f}$ and~$\tilde{\bgam}$ are constants, (\ref{PDE}) reduces to 
\begin{eqnarray} \label{CellPDE}
-\omega^2 \rho(\bx) \tilde{u} - \nabla_{\!\bk} \!\cdot\! \big( G(\bx)(\nabla _\bk\tilde{u} - \tilde\bgam)  \big) ~=~ \tilde{f} \quad \mbox{in~~} Y,
\end{eqnarray}
where $\nabla_{\!\bk}= \nabla + i \bk$. Here we note that (i)  $\tilde{f}$ and~$\tilde{\bgam}$ can be interpreted as the respective Fourier components of~$f$ and~$\bgamma$ at fixed wavenumber~$\bk$, and (ii) the appearance of eigenstrain~$\tilde{\bgam}$ helps guarantee the uniqueness of the Willis' homogenized description of~(\ref{CellPDE}), see~\cite{W11} for details. For completeness, the periodic boundary conditions accompanying~(\ref{CellPDE}) can be explicitly written as 
\begin{equation} \label{CellPDEBC1}
\begin{array}{rcl}
\tilde{u}|_{x_j=0} &\!\!\!=\!\!\!& \tilde{u}|_{x_j=\ell_j} \\*[2mm]
G(\nabla_{\!\bk}\tilde{u}-\tilde\bgam) \cdot \bnu|_{x_j=0} &\!\!\!=\!\!\!& -G(\nabla_{\!\bk}\tilde{u} - \tilde\bgam) \cdot \bnu|_{x_j=\ell_j}
\end{array}, \qquad j=\overline{1,d}
\end{equation}
where $x_j=\bx\!\cdot\!\be_j$ and $\bnu$ is the unit outward normal on~$\partial Y$. 

\subsection{Willis' effective description of the wave motion}

In the context of anti-plane shear waves, the respective expressions for strain, particle velocity, stress, and momentum density affiliated with~$\tilde{u}$ read 
\begin{equation}
\tilde{\beps} =\nabla_{\!\bk} \tilde{u}, \qquad \tilde{\textrm{\slshape{v}}} = -i \omega \tilde{u}, \qquad \tilde{\bsig} = G( \tilde{\beps} -\tilde\bgam), \qquad \tilde{p} = \rho \tilde{\textrm{\slshape{v}}}, \label{EffectiveRelations}
\end{equation}
which permit~(\ref{CellPDE}) to be rewritten as $-i\omega\tilde{p} -\!\nabla_{\!\bk} \!\cdot\! \tilde{\bsig} = \tilde{f}$. Thanks to~(\ref{CellPDEBC1}), averaging the latter over $Y$ yields the mean-fields equation
\begin{equation} \label{CellPDEEffectiveFieldsEquation}
-i\omega\langle\tilde{p}\rangle -i{\bk}\!\cdot\!\langle\tilde{\bsig}\rangle ~=~ \tilde{f},
\end{equation}
where $\langle\cdot\rangle$ denotes the $Y$-average of an $L^1(Y)$ function. In this setting, the goal is to obtain the counterpart of~\eqref{CellPDEEffectiveFieldsEquation} in terms of the mean motion~$\langle \tilde{u}\rangle$, and to explore its properties. This is accomplished in a consistent way~\cite{W80,W81a,W81b,W83,W84,W85} by introducing the Willis' effective constitutive relationship, which links the \emph{mean values} of the entries in~\eqref{EffectiveRelations} as 
\begin{eqnarray} \label{IntroEffectiveWillis}
\left[ \begin{array}{c}
 \langle \tilde{\bsig} \rangle\\
\langle \tilde{p} \rangle\end{array} \right] =
\left[ \begin{array}{cc}
\tilde{\boldsymbol{C}}^e (\bk,\omega)  & \tilde{\boldsymbol{S}}^{e\mbox{\tiny{1}}}(\bk,\omega) \\
\tilde{\boldsymbol{S}}^{e\mbox{\tiny{2}}}(\bk,\omega) & \tilde{\rho}^e (\bk,\omega) \end{array} \right] 
\left[ \begin{array}{c}
 \langle \tilde{\beps}  - \tilde\bgam \rangle\\
\vel \end{array} \right]. 
\end{eqnarray}
Here $\tilde{\boldsymbol{C}}^e$ and $\tilde{\rho}^e$ denote respectively the effective elasticity tensor and mass density, while~$\tilde{\boldsymbol{S}}^{e\mbox{\tiny{1}}}$ and~$\tilde{\boldsymbol{S}}^{e\mbox{\tiny{2}}}$ are the corresponding coupling vectors -- reflecting the non-local nature of the effective constitutive behavior. 

As examined in~\cite{NHA15}, an effective description of the mean wave motion via~\eqref{CellPDEEffectiveFieldsEquation} and~\eqref{IntroEffectiveWillis} makes sense only if the pair $(\bk,\omega)$ meets the \emph{necessary conditions} for homogenization in that 
\begin{equation}\label{necohom}
\bk \in \hat{Y}, \qquad \omega \,\lesssim\, \max_{\bx\in Y}\sqrt{\frac{G(\bx)}{\rho(\bx)}} \, \max_{\bk\in \hat{Y}}\|\bk\|, 
\end{equation}
where~$\hat{Y}\ni\boldsymbol{0}$ denotes the first Brioullin zone, given by the reciprocal of the unit cell~$Y$ in the Fourier $\bk$-space. In the context of~\eqref{PDE} and the plane-wave expansion approach, the first condition in~\eqref{necohom} implicitly requires that the Fourier spectrum of~$f(\bx)$ be restricted to~$\hat{Y}$. The above necessary conditions are schematically illustrated in Fig.~\ref{HomogenizableFigure}b) assuming~$d=1$ in~\eqref{PDE}, for which~$\bk=k$ and $\hat{Y}=\{k:|k|<\pi\}$. Depending on the local variation of the shear wave speed inside~$Y$, the second restriction in~\eqref{necohom} is such that the homogenizable region in the~$(\bk,\omega)$ space includes the acoustic branch and possibly the first optical branch, see~\cite{SN14,NHA15} for discussion. 
\begin{figure}
\centering\includegraphics[width=0.85\textwidth]{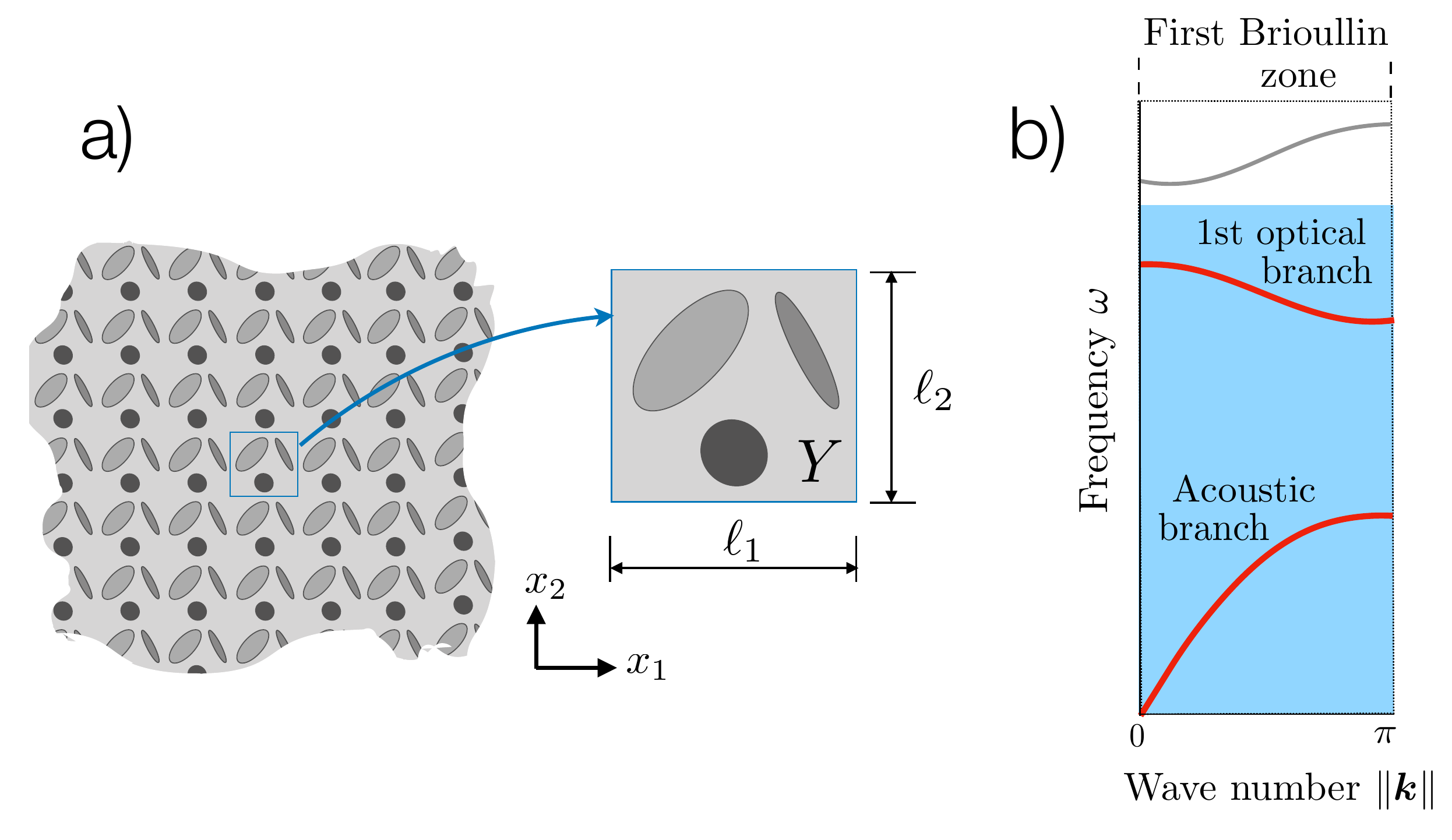} \vspace*{-2mm}
\caption{Homogenization of wave motion in periodic media: a) example of a periodic structure in~$\mathbb{R}^2$, and b) schematics of the homogenizable region (for a periodic structure in~$\mathbb{R}^1$) given by the shaded area in the $(\omega,\bk)$ space.} 
\label{HomogenizableFigure}
\end{figure}

A salient feature of the Willis' effective model~\eqref{CellPDEEffectiveFieldsEquation} and~\eqref{IntroEffectiveWillis} is that (barring a degenerate case to be examined later) the germane dispersion relationship~$\mathcal{D}^e(\bk,\omega)=0$, which permits non-trivial~$\langle\tilde{u}\rangle$ for $\tilde{f}=0$, recovers \emph{exactly}~\cite{NSK12} its antecedent~$\mathcal{D}(\bk,\omega)=0$ -- allowing for non-zero~$\tilde{u}$ when $\tilde{f}=0$ in~\eqref{CellPDE}. 

In principle, the suitability of~\eqref{IntroEffectiveWillis} as a mean-fields descriptor and the germane expressions for $\tilde{\boldsymbol{C}}^e\!, \tilde{\rho}^e, \tilde{\boldsymbol{S}}^{e\mbox{\tiny{1}}}$ and~$\tilde{\boldsymbol{S}}^{e\mbox{\tiny{2}}}$ are established by (i)~expressing~$\tilde{u}$ in~\eqref{CellPDE} via the Green's function for the unit cell~$Y$, and (ii) computing the Y-average of such result~\cite{NSK12,NHA15}. Typically, this leads to a complex spectral representation~\cite{NHA15, NSK12, W11} of the effective constitutive parameters that may exhibit instabilities when the pair $(\bk,\omega)$ resides on a Bloch branch in that~$\mathcal{D}^e(\bk,\omega)=0$. To deal with the problem, the authors in \cite{NSK12} for instance derive the Willis' model by invoking the Fourier series representation (akin to the Floquet-Bloch approach) and a regularization scheme where the Green's function is partitioned into a regular part and a singular component that diverges on a Bloch branch. 

In the sequel, we first propose an alternative representation of the Willis' model, using the eigensystem for the unit cell, that both (a) remains stable off and on effective Bloch branches, and (b) elucidates the aforementioned degenerate case where~$\mathcal{D}(\bk,\omega)=0$ but~$\mathcal{D}^e(\bk,\omega)\neq 0$. 

\section{Eigensystem representation of the Willis' model} \label{Representation}
To commence the analysis, we introduce the periodic function spaces
\begin{eqnarray*}
H^1_{p} (Y) &\! \!=\!\!& \{ g \in L^2(Y) \, : \, \nabla g \in (L^2(Y))^d, \,\,g|_{x_j=0} = g|_{x_j=\ell_j}, \,\, j=\overline{1,d},\},
\end{eqnarray*}
and the weighted Sobolev space $ L^2_\rho (Y) = \{ g: \int_Y \rho |g|^2 < \infty \}$. In this setting, let  $\tilde{w} \in H^1_p(Y)$ and $\tilde{\bv}=(\tilde{v}_1, \cdots, \tilde{v}_d) \in \big( H^1_p(Y) \big)^d$ denote the cell functions satisfying 
\begin{eqnarray} \label{CellPDEfw}
-\omega^2 \rho(\bx) \tilde{w}(\bx) - \nabla_{\!\bk}\!\cdot\!  \left( G(\bx) \nabla_{\!\bk} \tilde{w} \right) ~=~ 1 \quad \mbox{in} \quad Y,\\
-\omega^2 \rho(\bx) \tilde{\bv} (\bx) - \nabla_{\!\bk}\!\cdot\!  \Big( G(\bx) \big(\nabla_{\!\bk} \tilde{\bv}-  \boldsymbol{I} \big) \Big) ~=~ 0 \quad \mbox{in} \quad Y, \label{CellPDEfv}
\end{eqnarray}
subject to the boundary conditions
\begin{eqnarray} 
G \nabla_{\!\bk}\tilde{w}  \cdot \bnu|_{x_j=0} &\!\!=\!\!& -G\nabla_{\!\bk}\tilde{w} \cdot \bnu|_{x_j=\ell_j}, \qquad j=\overline{1,d}, \label{CellPDEfwBC} \\
\bnu \cdot G(\nabla_{\!\bk}\tilde{\bv}-\boldsymbol{I})|_{x_j=0} &\!\!=\!\!& -\bnu \cdot G(\nabla_{\!\bk}\tilde{\bv} - \boldsymbol{I})|_{x_j=\ell_j}, \qquad j=\overline{1,d}. \label{CellPDEfvBC}
\end{eqnarray}
Here~$\bI$ denotes the second-order identity tensor and, assuming hereon the Einstein summation notation, \mbox{$\nabla_{\!\bk}\hspace*{0.5pt} \bg \,=\, \be_j\otimes \partial\bg/\partial{x}_j +i\bk \otimes \bg\,$} for any vector or tensor field~$\bg$. We also remark that if one seeks a weak solution for $\tilde{u}$, $\tilde{w}$, or $\tilde{\bv}$ in a variational sense, the second of~(\ref{CellPDEBC1}), (\ref{CellPDEfwBC}), or~(\ref{CellPDEfvBC}) are implicitly included in such a variational formulation.

\begin{remark} 
Let $\mG(\bx;\by)$ denote the Green's function for the unit cell solving~\eqref{CellPDE}--\eqref{CellPDEBC1} with~$(\tilde{f},\tilde\bgamma)$ replaced by~$(\delta(\bx\!-\!\by),\bzero)$, and let~$\mH(\bx;\by)$ denote its dipole counterpart solving the same equations but with~$(\tilde{f},\tilde\bgamma)$ superseded by~$(\bzero,\bI\delta(\bx\!-\!\by))$. With such notation, the featured cell functions can be interpreted as the $Y$-averages $\tilde{w}=\langle\mG\rangle_{\by}$ and $\tilde{\bv}=\langle\mH\rangle_{\by}$, where the integration is performed over the source location~$\by\in Y$. 
\end{remark}

In what follows, the cell functions $\tilde{w}$ and $\tilde{\bv}$ are used as a ``basis'' for representing~$\tilde{u}$. Indeed, by the superposition argument one obtains the following lemma. 

\begin{lemma} \label{Baiswv}
Let~$\tilde{u}$ solve~\eqref{CellPDE}--\eqref{CellPDEBC1} where~$\tilde{f}$ and~$\tilde{\bgam}$ are constants. Then
\begin{eqnarray*}
\tilde{u} ~=~ \tilde{f}\hspace*{0.3pt} \tilde{w} \:+\: \tilde{\bgam} \!\cdot\! \tilde{\bv} \qquad \text{and} \qquad 
\langle\tilde{u}\rangle ~=~ \tilde{f}\hspace*{0.3pt} \langle\tilde{w}\rangle \:+\: \tilde{\bgam} \!\cdot\! \langle\tilde{\bv}\rangle.
\end{eqnarray*}
\end{lemma}
Next, we establish a representation of the Willis' effective model in terms of~$\tilde{w}$ and~$\tilde{\bv}$. To be precise, let
\begin{eqnarray}
\tilde{\boldsymbol{A}}(\bx)~=~  \frac{\langle \tilde{\bv} \rangle}{\langle \tilde{w} \rangle} \tilde{w}-\tilde{\bv}, \qquad 
\tilde{B}(\bx) ~=~ \frac{1}{i\omega} \Big(1 -\frac{1}{\langle \tilde{w} \rangle}\tilde{w}  \Big)+ \frac{1}{i
\omega}\Big(\frac{i\bk \cdot \langle \tilde{\bv} \rangle}{\langle \tilde{w} \rangle} \tilde{w}-i\bk \cdot \tilde{\bv} \Big). \label{Welldefinedness}
\end{eqnarray}
Accordingly one finds that  
\begin{eqnarray*}
\tilde{u} ~=~ \langle \tilde{u}\rangle \,+\, \tilde{\boldsymbol{A}} \!\cdot\! \langle \tilde{\beps}- \tilde{\bgam} \rangle \,+\, \tilde{B} \vel, 
\end{eqnarray*}
whereby
\begin{eqnarray*}
\langle \bsig \rangle &=& \langle G (\beps -\bgam) \rangle = \langle G\nabla_{\!\bk} \tilde{u} + G ( -\bgam)\rangle =  \langle G\nabla_{\!\bk} (\tilde{u}-\langle \tilde{u} \rangle) + G (i\bk \langle \tilde{u} \rangle -\bgam)\rangle \\
&=& \langle G \nabla_{\!\bk} \big( \tilde{\boldsymbol{A}} \!\cdot\! \langle  \tilde{\beps}- \tilde{\bgam} \rangle +  \tilde{B} \vel \big)  \rangle + \langle G \big( \langle \beps \rangle -\bgam \big) \rangle,
\end{eqnarray*}
and 
\begin{eqnarray*}
\langle \tilde{p} \rangle ~=\: -i\omega \langle \rho\tilde{u} \rangle ~=\:  -i\omega \big\langle \rho  \big( \langle \tilde{u}\rangle + \tilde{\boldsymbol{A}} \!\cdot\! \langle  \tilde{\beps}- \tilde{\bgam} \rangle +  \tilde{B} \vel \big) \big\rangle.
\end{eqnarray*}
Then the constitutive parameters in~\eqref{IntroEffectiveWillis} take the form 
\begin{eqnarray}
\tilde{\boldsymbol{C}}^e = \big\langle G\big(  \nabla_{\!\bk} \tilde{\boldsymbol{A}} + {\boldsymbol{I}} \big) \big\rangle, \label{EffectiveModuli} \quad
\tilde{\boldsymbol{S}}^{e\mbox{\tiny{1}}} =  \langle G \nabla_{\!\bk} \tilde{B}  \rangle, \quad
\tilde{\boldsymbol{S}}^{e\mbox{\tiny{2}}} =  \langle -i \omega \rho \tilde{\boldsymbol{A}}  \rangle, \quad
\tilde{\rho}^e =  \langle \rho -i \omega \rho \tilde{B}  \rangle. \label{EffectiveDensity}
\end{eqnarray}
From the expressions~(\ref{Welldefinedness}) for $\tilde{\boldsymbol{A}}(\bx)$ and $\tilde{B}(\bx)$, a direct calculation yields
\begin{eqnarray} 
\tilde{\rho}^e &=&  \tilde{Z}^e \langle \rho \tilde{w} \rangle \big(1- i\bk\!\cdot\!  \langle \tilde{\bv} \rangle   \big) + i\bk\!\cdot\!  \langle \rho \tilde{\bv} \rangle, \label{RelationsrhoRepre1} \\
\tilde{\boldsymbol{C}}^e &=&  \langle  G \rangle {\boldsymbol{I}} +  \tilde{Z}^e \langle G \nabla_{\!\bk} \tilde{w} \rangle \otimes \langle  \tilde{\bv} \rangle -\langle G \nabla_{\!\bk}   \tilde{\bv} \rangle , \\
\tilde{\boldsymbol{S}}^{e\mbox{\tiny{1}}} &=&  \frac{i\bk}{i\omega} \langle G \rangle - \frac{\tilde{Z}^e  }{i\omega}\langle G \nabla_{\!\bk}  \tilde{w} \rangle \big(1- i\bk\!\cdot\!  \langle \tilde{\bv} \rangle   \big) -\frac{1}{i\omega} \langle G \nabla_{\!\bk}  (i\bk \!\cdot\! \tilde{\bv}) \rangle , \\
\tilde{\boldsymbol{S}}^{e\mbox{\tiny{2}}} &=&  i \omega \big(  \langle \rho \tilde{\bv} \rangle-\tilde{Z}^e \langle \rho \tilde{w} \rangle \langle \tilde{\bv} \rangle\big), \label{RelationsSe2Repre1}
\end{eqnarray}
where~$\tilde{Z}^e$ is the so-called \emph{effective impedance} which recasts the mean-fields equation~\eqref{CellPDEEffectiveFieldsEquation} as $\tilde{Z}^e \langle\tilde{u}\rangle = \tilde{f}\,$ when $\bgamma=\bzero$; in particular, 
\begin{eqnarray} \label{EffectiveImpedanceEigen}
\tilde{Z}^e \langle\tilde{u}\rangle ~:=\: -i\omega\langle\tilde{p}\rangle - i{\bk}\!\cdot\!\langle\tilde{\bsig}\rangle|_{\bgamma=\bzero} 
\qquad \Longrightarrow\qquad  \tilde{Z}^e~=~ \frac{1}{\langle\tilde{w}\rangle}, 
\end{eqnarray}
noting that the second equality is a direct consequence of Lemma~\ref{Baiswv}.

\begin{remark}
Equations~(\ref{Welldefinedness}) hold when $\langle \tilde{w} \rangle \not = 0$. In fact when~$\tilde\bgamma=\bzero$, \eqref{CellPDEEffectiveFieldsEquation} and~\eqref{EffectiveImpedanceEigen} demonstrate that $\langle\tilde{u}\rangle=\langle \tilde{w}\rangle\tilde{f}$, whereby~$\langle \tilde{w} \rangle=0$ necessitates~$\langle \tilde{u} \rangle =0$. This is the case which the homogenization theory does not cover as examined in~\cite{NSK12}. From now on, to guarantee mathematical rigor and physical consistency we assume $\langle \tilde{w} \rangle \not = 0$.
\end{remark}


\begin{remark}
Since $\tilde{w} \in H^1_p(Y)$ and $\tilde{\bv} \in (H^1_p(Y))^d$, from~(\ref{RelationsrhoRepre1})--(\ref{RelationsSe2Repre1}) one concludes that the Willis' effective constitutive parameters are bounded. Furthermore these quantities are uniquely defined if (\ref{CellPDEfw}) and~(\ref{CellPDEfv}) each have a unique solution. 
For completeness, situations where the uniqueness of $\tilde{w}$ and $\tilde{\bv}$ does not hold are discussed in Section~\ref{Representation}\ref{NonDegenerateCaseSubsection}. 
\end{remark}

\subsection{Eigensystem for the unit cell of periodicity}\label{SectionEigen}
From the variational formulation one can show that $\left( - \nabla_{\!\bk} \!\cdot\! \left( G(\bx) \nabla_{\!\bk}  \right) \right)^{-1}$, as an operator from $L^2_\rho (Y)$ to itself with the range in $H^1_{p}(Y)$ subject to appropriate boundary conditions, is a compact self-adjoint operator \cite{PBL78}. Hence for each~$\bk$ there exists an eigensystem $\{ \tilde{\phi}_{m}, \tilde{\lambda}_{m}\}$ that satisfies
\begin{eqnarray} \label{Eigensystem}
- \nabla_{\!\bk} \!\cdot\! \big( G(\bx) \nabla_{\!\bk} \tilde{\phi}_{m} \big) &\!\!=\!\!& \tilde{\lambda}_{m}  \rho \, \tilde{\phi}_m \quad \mbox{in} \quad Y, \\
G \nabla_{\!\bk}\tilde{\phi}_{m}  \cdot \bnu|_{x_j=0} &\!\!=\!\!& -G\nabla_{\!\bk}\tilde{\phi}_m \cdot \bnu|_{x_j=\ell_j}, \quad j=\overline{1,d} \nonumber
\end{eqnarray}
where~$\tilde\lambda_m\in\mathbb{R}\,$, $\,\tilde{\phi}_m \in H^1_{p} (Y)$, and $\{ \tilde{\phi}_{m} \}$ are complete and orthonormal in $L^2_\rho (Y)$, i.e. 
\begin{eqnarray*}
\int_Y \rho(\bx) \tilde{\phi}_m(\bx) \overline{ \tilde{\phi}}_n(\bx)\, \textrm{d}\bx ~=~ \delta_{mn}.
\end{eqnarray*}
Since $\tilde{u}(\bx)$ satisfies~(\ref{CellPDE}), one obtains the variational formulation 
\begin{multline} \label{CellPDEEigenstrainVariational}
-\omega^2 \int_Y \rho(\bx) \tilde{u}(\bx) \overline{ \tilde{\phi}}_m(\bx) \,\textrm{d}\bx \;-\; \int_Y \tilde{u}(\bx)\overline{  \nabla_{\!\bk} \cdot \left( G(\bx) \nabla_{\!\bk}  \tilde{\phi}_m(\bx) \right) }\,\textrm{d}\bx  \\
=~\, \int_Y G(\bx) \overline{ \nabla_{\!\bk} \tilde{\phi}_m(\bx)} \,\cdot\, \tilde{\bgam}\, \textrm{d}\bx  \;+\;  
\int_Y \tilde{f} \, \overline{ \tilde{\phi}}_m(\bx) \,\textrm{d}\bx.  
\end{multline}
Thanks to the completeness of $\{ \tilde{\phi}_m \}$ in $L^2_\rho (Y)$, any $\tilde{u}\in L^2_\rho (Y)$ can be written as
\begin{eqnarray*}
\tilde{u}(\bx) ~=~ \sum_{m=1}^\infty \tilde{\alpha}_m \tilde{\phi}_m(\bx) \quad \mbox{in} \quad L^2_\rho (Y),
\end{eqnarray*}
where $\tilde{\alpha}_m$ are to be determined. By~\eqref{CellPDEEigenstrainVariational} and the orthogonality of $\{\tilde{\phi}_m \}$ in $L^2_\rho (Y)$, one further has 
\begin{eqnarray*}
-\omega^2 \tilde{\alpha}_m + \tilde{\lambda}_m \tilde{\alpha}_m ~=~ (\tilde{f}  ,  \tilde{\phi}_m) + (G  \tilde{\bgam}, \nabla_{\!\bk}  \tilde{\phi}_m),
\end{eqnarray*}
where $(\cdot , \cdot)$ denotes the usual $L^2(Y)$ inner product. This demonstrates that
\begin{eqnarray}\label{CellPDEukalpha}
\tilde{\alpha}_m ~=~ \frac{(\tilde{f},\tilde{\phi}_m)}{\tilde{\lambda}_m -\omega^2 } \,+\, \frac{ (G  \tilde{\bgam}, \nabla_{\!\bk}  \tilde{\phi}_m)}{\tilde{\lambda}_m -\omega^2 } \quad \mbox{if} \quad  \omega^2 \not= \tilde{\lambda}_m, \quad m\in\mathbb{Z}^{+},
\end{eqnarray}
where $\mathbb{Z}^{+}$ denotes the set of positive integers. On recalling that~$\tilde{f}$ and ~$\tilde{\bgam}$ are constants, one finds from~(\ref{CellPDEukalpha}) that the expressions
\begin{eqnarray}
\tilde{w} &=& \sum_{m=1}^\infty \frac{(1 ,  \tilde{\phi}_m)}{\tilde{\lambda}_m -\omega^2 } \tilde{\phi}_m, \quad \mbox{when} \quad \omega^2 \not= \lambda_m,  \quad m\in\mathbb{Z}^{+} \label{Eigenexpansionw} \\
\tilde{\bv} &=&\sum_{m=1}^\infty \frac{(G, \nabla_{\!\bk} \tilde{\phi}_m)}{\tilde{\lambda}_m -\omega^2 } \tilde{\phi}_m, \quad \mbox{when} \quad \omega^2 \not= \lambda_m,  \quad m\in\mathbb{Z}^{+} \label{Eigenexpansionv}
\end{eqnarray}
hold in the $L^2_\rho(Y)$ sense. 

\begin{remark}
From~\eqref{EffectiveImpedanceEigen} and~\eqref{Eigenexpansionw}, it is easy to see that~$\langle\tilde{w}\rangle$ and thus the effective impedance~$\tilde{Z}^e$ are real-valued. 
\end{remark}

\begin{remark}\label{rem5}
From the above arguments, one finds that the necessary and sufficient condition that~(\ref{CellPDEfw}) and~(\ref{CellPDEfv}) each have a unique solution is $\omega^2 \not= \lambda_m$, $\forall m$. If $\omega^2 = \lambda_n$ for some $n$, then~(\ref{CellPDEfw}) is still solvable provided $(1,\tilde{\phi}_{n})=0$ and~(\ref{CellPDEfv}) is still solvable provided $(G, \nabla_{\!\bk}\tilde{\phi}_{n})=\bzero$. These conditions are hereon referred as the solvability conditions.
\end{remark}

We next establish the representation of~$\tilde{w}$ and~$\tilde{\bv}$ assuming that the above solvability conditions hold for some $\tilde{\lambda}_n$ ($n$ fixed). For generality, let $\tilde{\lambda}_n$ be either a simple or repeated eigenvalue, and denote by  
\begin{equation}\label{multeig}
\{\tilde{\phi}_{j}\}_{j\in \Lambda_n}, \qquad \Lambda_n=\{j:n\!-\!M_n\leqslant j \leqslant n\!+\!N_n\}, \qquad M_n, N_n \geqslant 0 
\end{equation} 
the set of eigenfunctions corresponding to $\tilde{\lambda}_{n}$.  Further let $V_n$ be the closure of the space spanned by this basis, and let $V_n^\perp$ be the orthogonal complement to $V_n$ in the periodic $L^2(Y)$ space. Now we assume that $\langle \tilde{\phi}_{j} \rangle = 0$ and $(G, \nabla_{\!\bk} \tilde{\phi}_{j})=\bzero$ for all $j\in \Lambda_n$, i.e. that the solvability conditions for $\tilde{w}$ and $\tilde{\bv}$ hold. This yields the eigenfunction representation 
\begin{eqnarray}
\tilde{w} &\!\!=\!\!\!& \sum_{\Lambda_n\not\ni m=1}^{\infty}  \frac{(1,\tilde{\phi}_{m})}{\tilde{\lambda}_{m}-\omega^2} \tilde{\phi}_{m} \quad \mbox{when} \quad \omega^2 \neq \lambda_m, \quad m\in\mathbb{Z}^{+}\backslash\{n\}, \label{EigenexpansionwSolv} \\
\tilde{\bv} &\!\!=\!\!\!&\sum_{\Lambda_n\not\ni m=1}^{\infty} \frac{(G, \nabla_{\!\bk} \tilde{\phi}_{m})}{\tilde{\lambda}_{m} - \omega^2} \tilde{\phi}_{m} \quad \mbox{when} \quad \omega^2 \neq \lambda_m, \quad m\in\mathbb{Z}^{+}\backslash\{n\},\label{EigenexpansionvSolv}
\end{eqnarray}
which is, at $\omega^2=\tilde{\lambda}_n$, bounded and unique \emph{up to} a free-wave contribution in~$V_n$ whose basis solves~\eqref{Eigensystem} when $\tilde{\lambda}_m=\tilde{\lambda}_n$. Due to the fact that $\langle\tilde{\phi}_j\rangle=0$ for all~$j\in\Lambda_n$, however, the averages~$\langle\tilde{w}\rangle$ and~$\langle\tilde{\bv}\rangle$ are both bounded and unique at $\omega^2=\tilde{\lambda}_n$. More generally they are, for given~$\bk$, continuous functions of~$\omega$ over any closed interval containing~$\tilde{\lambda}_n^{1/2}$ but not (the square roots of) other eigenvalues.

\subsection{Properties of the effective constitutive parameters}

In this section we shed light on the effective constitutive parameters~(\ref{RelationsrhoRepre1})--(\ref{RelationsSe2Repre1}), written in terms of~$\tilde{w}$ and~$\tilde{\bv}$, assuming that $\omega \not= \tilde{\lambda}_n$ for all $n$. To this end, we need the following two lemmas and we refer to Section~\ref{App} for their proofs. 
\begin{lemma} \label{ReciprocityLemma}
The Hermitian symmetry 
\begin{eqnarray*}
\langle G \nabla_{\!\bk} \tilde{\bv} \rangle ~=~ \langle G \nabla_{\!\bk} \tilde{\bv} \rangle^*
\end{eqnarray*} 
holds, where $(\cdot)^*$ denotes the conjugate transpose.
\end{lemma}
\begin{lemma} \label{RelationsLemma}
The following equations hold
\begin{eqnarray*}
i\bk \cdot \langle \tilde{\bv} \rangle &=& 1+\omega^2 \overline{\langle \rho \tilde{w} \rangle}, \\
 \langle G \nabla_{\!\bk} (\tilde{\bv} \cdot i\bk) \rangle &=& i\bk \langle G\rangle + \omega^2 \overline{\langle \rho \tilde{\bv} \rangle}, \\
\langle G \nabla_{\!\bk} \tilde{w} \rangle &=& \overline{\langle \tilde{\bv} \rangle}.
\end{eqnarray*}
\end{lemma}

The Willis' effective model can now be recast in terms of~$\tilde{w}$ and~$\tilde{\bv}$ as follows. 
\begin{Prop} \label{ThmRelationsEffectiveConstitutiveRelations}
The effective constitutive parameters  $\tilde{\boldsymbol{C}}^e$,  $\tilde{\rho}^e$,  $\tilde{\boldsymbol{S}}^{e\mbox{\tiny{1}}}$ and  $\tilde{\boldsymbol{S}}^{e\mbox{\tiny{2}}}$ carry the symmetries 
\begin{eqnarray*}
\tilde{\rho}^e = \tilde{\rho}^{e*}, \qquad \tilde{\boldsymbol{C}}^e\! = \tilde{\boldsymbol{C}}^{e*}, \qquad  
\tilde{\boldsymbol{S}}^{e\mbox{\tiny{1}}}\!=-\tilde{\boldsymbol{S}}^{e\mbox{\tiny{2}}*},
\end{eqnarray*}
and admit the eigensystem representation  
\begin{eqnarray} \label{Relationsrho}
\tilde{\rho}^e &=&  -\omega^2 \tilde{Z}^e |\langle \rho \tilde{w} \rangle|^2   + i\bk\cdot  \langle \rho \tilde{\bv} \rangle, \\
\tilde{\boldsymbol{C}}^e &=&  \langle  G \rangle {\boldsymbol{I}} +  \tilde{Z}^e \,\overline{\langle\tilde{\bv}\rangle} \otimes \langle  \tilde{\bv} \rangle -\langle G \nabla_{\!\bk} \tilde{\bv} \rangle , \label{RelationsC} \\
\tilde{\boldsymbol{S}}^{e\mbox{\tiny{2}}} &=&  i \omega \big(  \langle \rho \tilde{\bv} \rangle-\tilde{Z}^e \langle \rho \tilde{w} \rangle \langle \tilde{\bv} \rangle\big), \label{RelationsS}
\end{eqnarray}
where~$\tilde{w}$ and~$\tilde{\bv}$ are given by~\eqref{Eigenexpansionw}--\eqref{Eigenexpansionv}. 
\end{Prop}
\begin{proof}
Let us first recall~(\ref{EffectiveImpedanceEigen}), stating that $\langle \tilde{w} \rangle^{-1}=\tilde{Z}^e$, and representation~(\ref{RelationsrhoRepre1})--(\ref{RelationsSe2Repre1}) of the effective constitutive parameters. From Lemma~\ref{RelationsLemma}, we immediately have that equations (\ref{Relationsrho})--(\ref{RelationsS}) hold. It remains to show that
$\,\tilde{\rho}^e = \tilde{\rho}^{e*}\!$, $\,\tilde{\boldsymbol{C}}^e=  \tilde{\boldsymbol{C}}^{e*}\!$, and $\,\tilde{\boldsymbol{S}}^{e\mbox{\tiny{2}}*}=-\tilde{\boldsymbol{S}}^{e\mbox{\tiny{1}}}$.

Using the eigenfunction expansion~\eqref{Eigenexpansionv} of $\tilde{\bv}$ and the divergence theorem, one finds 
\begin{eqnarray*}
i\bk\cdot  \langle \rho \tilde{\bv} \rangle &\!\!=\!\!& \sum_{m=1}^\infty  \frac{ (i\bk, G \nabla_{\!\bk}\tilde{\phi}_m) \langle \rho \tilde{\phi}_m \rangle}{\tilde{\lambda}_m - \omega^2}  ~=~ \sum_{m=1}^\infty  \frac{ \tilde{\lambda}_m (1,\rho \tilde{\phi}_m) \langle \rho \tilde{\phi}_m \rangle}{\tilde{\lambda}_m - \omega^2} \,\in\, \R, 
\end{eqnarray*}
whereby~$\tilde{\rho}^e = \tilde{\rho}^{e*}$. From Lemma \ref{RelationsLemma}, on the other hand, it follows that 
\begin{eqnarray*}
\tilde{\boldsymbol{S}}^{e\mbox{\tiny{1}}} &\!\!\!=\!\!&  \frac{i\bk}{i\omega} \langle G \rangle - i\omega \tilde{Z}^e   \overline{ \langle   \tilde{\bv} \rangle}   \overline{ \langle \rho \tilde{w} \rangle}   -\frac{1}{i\omega} (i\bk \langle G\rangle + \omega^2 \overline{\langle \rho \tilde{\bv} \rangle} ) ~\;=\: 
- i\omega \tilde{Z}^e   \overline{ \langle   \tilde{\bv} \rangle}   \overline{ \langle \rho \tilde{w} \rangle}  + i\omega \overline{\langle \rho \tilde{\bv} \rangle},
\end{eqnarray*}
so that~$\tilde{\boldsymbol{S}}^{e\mbox{\tiny{1}}}=- \tilde{\boldsymbol{S}}^{e\mbox{\tiny{2}}*}$ since~$\tilde{Z}^e$ is real-valued. The remaining claim that $\tilde{\boldsymbol{C}}^e\!=\tilde{\boldsymbol{C}}^{e*}$ is a direct consequence of~\eqref{RelationsC} and Lemma~\ref{ReciprocityLemma}. 
\end{proof} \proofend

\subsection{Effective impedance and dispersion relationship}
In physical terms, the effective impedance~$\tilde{Z}^e$ synthesizes the linear operator acting on~$\langle\tilde{u}\rangle$  in the balance of linear momentum~(\ref{CellPDEEffectiveFieldsEquation}) when~$\bgamma=\bzero$. Its relationship with~$\tilde{\rho}^e$, $\tilde{\boldsymbol{C}}^e$ and $\tilde{\boldsymbol{S}}^{e\mbox{\tiny{2}}}$ is established via the following result, see also~\cite{NSK12,NHA16}.
\begin{lemma}
The effective impedance~\eqref{EffectiveImpedanceEigen} can be written in terms of the effective constitutive parameters as
\begin{eqnarray} \label{EffectiveImpedancebyEffectiveRelations}
\tilde{Z}^e(\bk,\omega) ~=~ -i \bk \cdot \tilde{\boldsymbol{C}}^e \!\!\cdot i \bk \;-\; i\bk\!\cdot\!(\tilde{\boldsymbol{S}}^{e\mbox{\tiny{2}}}\!\!+\tilde{\boldsymbol{S}}^{e\mbox{\tiny{2}}*}) i\omega \;-\; \omega^2 \tilde{\rho}^e.
\end{eqnarray}
\end{lemma}
\begin{proof}
Substituting the Willis' constitutive relationship~(\ref{IntroEffectiveWillis}) into the balance of linear momentum~\eqref{CellPDEEffectiveFieldsEquation} with~$\bgamma=\bzero$ yields 
\begin{eqnarray*}
-i\omega \big(\tilde{\boldsymbol{S}}^{e\mbox{\tiny{2}}} \!\!\cdot \langle \tilde{\beps}\rangle +\tilde{\rho}^e \vel \big) \,-\, 
i{\bk}\cdot \big(\tilde{\boldsymbol{C}}^e \!\!\cdot\!\langle\tilde{\beps}\rangle + \tilde{\boldsymbol{S}}^{e\mbox{\tiny{1}}} \vel \big) ~=~ \tilde{f}.
\end{eqnarray*}
From~(\ref{EffectiveRelations}), however, one has~$\langle \tilde{\beps} \rangle = i \bk \langle \tilde{u} \rangle$ and~$\vel = -i \omega \langle \tilde{u} \rangle$ which immediately recovers~\eqref{EffectiveImpedancebyEffectiveRelations} since~$\tilde{\boldsymbol{S}}^{e\mbox{\tiny{1}}}\!=-\tilde{\boldsymbol{S}}^{e\mbox{\tiny{2}}*}$ due to Proposition~\ref{ThmRelationsEffectiveConstitutiveRelations}.
\end{proof} 

To expose the dispersive characteristics of the homogenized system, one finds from~\eqref{CellPDEEffectiveFieldsEquation} and~\eqref{EffectiveImpedanceEigen} that the existence of free waves requires a non-trivial solution to $\tilde{Z}^e\langle\tilde{u}\rangle=0$, giving the \emph{effective} dispersion equation as $\mathcal{D}^e(\bk,\omega)\equiv\tilde{Z}^e(\bk,\omega) = 0$. On the other hand, eigensystem~\eqref{Eigensystem} of the original problem~\eqref{CellPDE} demonstrates that the \emph{exact} dispersion equation $\mathcal{D}(\bk,\omega)=0$ is solved by the Bloch pairs $\{\big(\bk, \tilde{\lambda}_n^{1/2}(\bk))\}_{n=1}^\infty$,  where $(\bk,\tilde{\lambda}_1^{1/2}(\bk))$ in particular specifies the so-called acoustic branch. As examined in~\cite{NSK12}, these two statements of the dispersion relationship are \emph{equivalent} barring the following situations: 
\vspace*{-3mm}
\begin{itemize}
\item the case where at least one Bloch wave mode~$\tilde{u}_B$ has zero mean, $\langle \tilde{u}_B\rangle=0$. In the context of~\eqref{Eigensystem}, this happens when $\omega=\tilde{\lambda}_n^{1/2}(\bk)$ so that $\langle \tilde{u}_B\rangle\equiv\langle\tilde{\phi}_n\rangle(\bk)=0$. Since such eigenmodes are not observable from the effective i.e. macroscopic point of view, this situation is not covered by the homogenization theory. \\
\vspace*{-2mm}
\item the instance of intersecting Bloch wave branches or double points, mathematically corresponding to the occurence of repeated eigenvalues in~\eqref{Eigensystem}.
\end{itemize}

\noindent To help better understand the second case, denote by~\eqref{multeig} the eigenfunctions corresponding to eigenvalue $\tilde{\lambda}_{n}$ with multiplicity $M_n\!+N_n+1$. When $\langle \tilde{\phi}_{j} \rangle = 0$ for all $j\in\Lambda_n$, \eqref{EigenexpansionwSolv} demonstrates that~$\langle \tilde{w} \rangle$ remains bounded when $\omega^2 = \tilde{\lambda}_n$, whereby the effective impedance fails to capture the Bloch pair $(\bk,\tilde{\lambda}_n^{1/2})$. On the other hand if there exists $j\in\Lambda_n$ so that $\langle \tilde{\phi}_{j} \rangle \not = 0$, one has 
\begin{eqnarray*}
\langle \tilde{w} \rangle ~=  \sum_{\Lambda_n\ni j, \langle \tilde{\phi}_{j} \rangle \not=0} \frac{(1,\tilde{\phi}_{j})}{\tilde{\lambda}_n -\omega^2 } \langle\tilde{\phi}_{j}\rangle \;+\; O(1) ~\to~ \infty \qquad \text{as}\quad \omega^2 \to \tilde{\lambda}_n,
\end{eqnarray*}
and consequently~$\tilde{Z}^e=\langle \tilde{w} \rangle^{-1} \to 0\,$ as  $\,\omega^2 \to \tilde{\lambda}_n$. Hence the effective impedance does capture the Bloch pair $(\bk, \tilde{\lambda}_n^{1/2})$ in such situations.

Summarizing the above arguments, we have the following theorem. For generality, we allow for $M_n=N_n=0$ in~\eqref{multeig} as to include both simple and repeated eigenvalues. 
\begin{theorem}
Assume that for given~$\bk$, $\lambda_n$ is an eigenvalue of~\eqref{Eigensystem} corresponding to eigenfunction(s)~\eqref{multeig}. Then the effective impedance~$\tilde{Z}^e$ given by~\eqref{EffectiveImpedancebyEffectiveRelations} is capable of capturing the dispersion pair $(\bk,\tilde{\lambda}_n^{1/2})$ if there exists $j\in\Lambda_n$ such that $\langle\tilde{\phi}_{j}\rangle \neq 0$. Further this dispersion pair is identifiable  by~$\tilde{Z}^e$ as a double point only if there are multiple eigenfunctions~$\tilde{\phi}_{j}$, $j\in\Lambda_n$ with non-zero mean.
\end{theorem}
In the sequel, we refer to the situation $\sum_{j\in\Lambda_n}|\langle\tilde{\phi}_{j}\rangle|>0\,$ (resp. $\sum_{j\in\Lambda_n}|\langle\tilde{\phi}_{j}\rangle|=0)$ as ``visible'' (resp. ``invisible'') case in the sense of detection of the dispersion pair $(\bk,\tilde{\lambda}_n^{1/2})$ by~$\tilde{Z}^e(\bk,\omega)=0$.

\subsection{Wavenumber-frequency behavior of the effective constitutive relations} \label{NonDegenerateCaseSubsection}

Willis' effective constitutive relations are typically derived using the Green's function for the unit cell~\cite{NHA15, NSK12, W11}, which requires a closer examination when (for given~$\bk$) $\,\omega^2\to\tilde{\lambda}_n(\bk)$. To this end, the authors in~\cite{NSK12} for example introduce a finite-dimensional (Fourier series) approximation of~\eqref{CellPDE}--\eqref{CellPDEBC1} and partition of the Green function into a regular and diverging part as $\omega\to\tilde{\lambda}^{1/2}_n$. In this section, we study the limiting behavior of the effective constitutive parameters when $\omega\to\tilde{\lambda}^{1/2}_n$ using the eigensystem~\eqref{Eigensystem} for the unit cell.

As can be seen from~\eqref{RelationsrhoRepre1}--\eqref{RelationsSe2Repre1}, the effective constitutive relations involve terms~$G\nabla_{\!\bk} \tilde{w}$ and $G\nabla_{\!\bk} \tilde{\bv}$. However the expressions~\eqref{Eigenexpansionw}--\eqref{Eigenexpansionv} for~$\tilde{w}$ and~$\tilde{\bv}$ hold in the $L^2(Y)$ sense, and their gradients may not be computable using term-by-term differentiation.  In order to obtain a rigorous eigenfunction expansion of~$G\nabla_{\!\bk} \tilde{w}$ and~$G\nabla_{\!\bk} \tilde{\bv}$, we need the following lemma and we refer to Appendix, Section~\ref{App} for its proof.

\begin{lemma} \label{LemmaCellBasis}
Assume that $\bzeta \in \big( H^1_p(Y) \big)^d$ satisfies the unit cell problem 
\begin{eqnarray} \label{CellBasis}
\nabla_{\!\bk} \!\cdot\! \big(G(\bx)(\nabla_{\!\bk}\bzeta -\boldsymbol{I}) \big)&\!\!=\!\!&0 \quad \mbox{in} \quad Y, \\
\bnu \cdot G (\nabla_{\!\bk} \bzeta -\boldsymbol{I} )|_{x_j = 0} &\!\!=\!\!& -\bnu \cdot G (\nabla_{\!\bk} \bzeta -\boldsymbol{I} )|_{x_j = \ell_j},  \quad j=\overline{1,d}. \nonumber
\end{eqnarray}
Then
\begin{eqnarray*}
\langle G \nabla_{\!\bk} \tilde{w} \rangle &\!\!=\!\!& \langle \overline{\bzeta} \rangle + \omega^2 \langle \rho \tilde{w} \overline{\bzeta} \rangle,  \\
\langle G \nabla_{\!\bk} \tilde{\bv} \rangle &\!\!=\!\!& \langle G \overline{\nabla_{\!\bk} \bzeta} \rangle\tran + \omega^2 \langle \rho  \overline{\bzeta} \otimes \tilde{\bv} \rangle,
\end{eqnarray*}
where $(\cdot)\tran$ denotes tensor or vector transpose.
\end{lemma}

Lemma~\ref{LemmaCellBasis} computes the averages of~$G\nabla_{\!\bk} \tilde{w}$ and~$G\nabla_{\!\bk} \tilde{\bv}$ in terms of $\tilde{w}$, $\tilde{\bv}$ and ${\bzeta}$. However since~$\bzeta$ is independent of~$\omega$ and thus~$\tilde{\lambda}_n$, the study of $\langle G \nabla_{\!\bk} \tilde{w} \rangle$ and $\langle G \nabla_{\!\bk} \tilde{\bv} \rangle\,$ as $\,\omega\to\tilde{\lambda}^{1/2}_n$ is reduced to that of~$\tilde{w}$ and~$\tilde{\bv}$.

\subsubsection{Invisible case}\label{EffectiveCase2}

It was shown earlier that when~$\langle \tilde{\phi}_{j} \rangle =0$ for all~$j\in\Lambda_n$, $\langle\tilde{w}\rangle$ is a continuous function of~$\omega$ over any sufficiently small neighborhood of~$\tilde{\lambda}^{1/2}_n$ thanks to~\eqref{EigenexpansionwSolv}. If further $(G,\nabla_{\!\bk}\tilde{\phi}_{j})=\bzero$ for all~$j\in\Lambda_n$, then~\eqref{EigenexpansionvSolv} applies and~$\langle\rho\tilde{\phi}_j\rangle=0$ due to~\eqref{Eigensystem}, whereby~$\langle \tilde{\bv} \rangle$, $\langle \rho \tilde{w} \rangle$ and~$\langle \rho \tilde{\bv} \rangle$ are also continuous functions of~$\omega$ near~$\tilde{\lambda}_n^{1/2}$. Thanks to Lemma~\ref{LemmaCellBasis}, the same claim applies to~$\langle G \nabla_{\!\bk} \tilde{w} \rangle$ and~$\langle G \nabla_{\!\bk} \tilde{\bv} \rangle$, whereby $\tilde{\boldsymbol{C}}^e\!$,  $\tilde{\rho}^e$,  $\tilde{\boldsymbol{S}}^{e\mbox{\tiny{1}}}$ and  $\tilde{\boldsymbol{S}}^{e\mbox{\tiny{2}}}$ in~\eqref{RelationsrhoRepre1}--\eqref{RelationsSe2Repre1} are continuous functions of~$\omega$ over any closed interval containing~$\tilde{\lambda}_n^{1/2}$ but not (the square roots of) other eigenvalues. This situation is related to the so-called degenerate case discussed in~\cite{NSK12}. Here we finally remark that: (i) when $G(\bx)=\text{const.}$, $\langle \tilde{\phi}_{j} \rangle =0$ guarantees that $(G,\nabla_{\!\bk}\tilde{\phi}_{j})=\bzero$, and (ii) if $(G, \nabla_{\!\bk} \tilde{\phi}_{j})\not=\bzero$ for some~$j\in\Lambda_n$, the effective constitutive parameters may not be uniquely defined when $\lambda_n = \omega^2$. This can be seen, for instance, in the case where~$\tilde{\lambda}_n$ has multiplicity one.

\subsubsection{Visible case} \label{EffectiveCase1}

In situations where~$\langle \tilde{\phi}_{j} \rangle \not=0$ for some~$j\in\Lambda_n$, from the eigenfunction expansions (\ref{Eigenexpansionw})--(\ref{Eigenexpansionv}) of $\tilde{w}$ and $\tilde{\bv}$, one finds assuming~$\,\omega^2\!-\tilde{\lambda}_n =o(1)$ that 
\begin{eqnarray*}
\tilde{w} ~=\!\!  \sum_{\Lambda_n\ni j, \langle \tilde{\phi}_{j} \rangle \not=0} \frac{(1 ,  \tilde{\phi}_{j})}{\tilde{\lambda}_{n} -\omega^2 } \tilde{\phi}_{j} \;+\; O(1), 
\qquad \mbox{and} \qquad
\tilde{\bv} ~=\!\! \sum_{\Lambda_n\ni j, \langle G, \nabla_{\!\bk} \tilde{\phi}_{j} \rangle \not=\bzero} \frac{(G,  \nabla_{\!\bk} \tilde{\phi}_{j})}{\tilde{\lambda}_{n} -\omega^2 } \tilde{\phi}_{j} \;+\; O(1).
\end{eqnarray*}

Next we pursue a detailed analysis when the eigenvalue $\tilde{\lambda}_n$ has multiplicity one. Here the corresponding eigenfunction is~$\tilde{\phi}_{n}$, and it is further assumed that $\langle G, \nabla_{\!\bk} \tilde{\phi}_{n} \rangle \not=\bzero$. We first note from the above expression that 
$\langle \tilde{w} \rangle \to \infty$ and $\langle \tilde{\bv} \rangle \to \infty\,$ as $\,\omega^2 \to \tilde{\lambda}_n$. To prove that~$\tilde{\boldsymbol{C}}^e\!$,  $\tilde{\rho}^e$,  $\tilde{\boldsymbol{S}}^{e\mbox{\tiny{1}}}$ and  $\tilde{\boldsymbol{S}}^{e\mbox{\tiny{2}}}$ remain well-defined in this case, it is sufficient to show that the germane singularities cancel. For brevity, we focus on the analysis of~$\tilde{\rho}^e$. From~\eqref{RelationsrhoRepre1}, one has 
\begin{eqnarray*}
\tilde{\rho}^e ~=~ \frac{\langle\rho\tilde{w}\rangle}{\langle\tilde{w}\rangle} \big(1- i\bk\!\cdot\!  \langle \tilde{\bv} \rangle   \big) + i\bk\!\cdot\!  \langle \rho \tilde{\bv} \rangle,  
\end{eqnarray*}
while~(\ref{Eigenexpansionw}) and~(\ref{Eigenexpansionv}) demonstrate that  
\begin{eqnarray*}
\tilde{w} ~=~  \frac{(1,\tilde{\phi}_{n})}{\tilde{\lambda}_{n} -\omega^2 } \tilde{\phi}_{n} \;+\; O(1) \qquad \mbox{and} \qquad
\tilde{\bv} ~=~ \frac{(G,\nabla_{\!\bk} \tilde{\phi}_{n})}{\tilde{\lambda}_{n} -\omega^2} \tilde{\phi}_{n} \;+\; O(1)
\end{eqnarray*}
when  $\omega^2\!- \tilde{\lambda}_n =o(1)$ and $\Lambda_n=\{n\}$. A direct calculation then shows that 
\[
\lim_{\omega\to\tilde{\lambda}_n^{1/2}}\tilde{\rho}^e(\bk,\omega) ~=~ O(1) \quad \text{for fixed}~\bk.
\]
A similar calculation, aided by~Lemma~\ref{LemmaCellBasis}, can be performed to show that $\,\tilde{\boldsymbol{S}}^{e\mbox{\tiny{1}}}\!$, $\tilde{\boldsymbol{S}}^{e\mbox{\tiny{2}}}$ and $\tilde{\boldsymbol{C}}^e$ likewise remain bounded when $\omega\to\tilde{\lambda}_n^{1/2}$. When~$\langle G,\nabla_{\!\bk}\tilde{\phi}_{n}\rangle=\bzero$, on the other hand, $\tilde{\bv}$ does not allow for a unique representation, implying that the Willis' effective constitutive parameters in~\eqref{RelationsrhoRepre1}--\eqref{RelationsSe2Repre1} are possibly non-unique in this case. Finally we remark that if~$\tilde{\lambda}_n$ has multiplicity larger than $1$, following a similar argument, one can investigate the more complicated behavior of~$\tilde{\rho}^e$, $\tilde{\boldsymbol{C}}^e\!$, $\tilde{\boldsymbol{S}}^{e\mbox{\tiny{1}}}$ and  $\tilde{\boldsymbol{S}}^{e\mbox{\tiny{2}}}$ as $\omega^2 \to \tilde{\lambda}_n$, see also the discussion of the so-called exceptional case in \cite{NSK12}.

\section{LW-LF approximation of the Willis' model and comparison with the two-scale homogenization result} \label{ComparisonWillisTwoscale}

\subsection{Effective impedance obtained by two-scale homogenization}
In principle, the two-scale homogenization approach~\cite{PBL78} can be used to approximate the acoustic branch, $\omega=\tilde{\lambda}_1^{1/2}(\bk)$, of the dispersion relationship at long wavelengths where $\|\bk\|\ll |Y|^{-1/d}$. Recently, such an asymptotic approach was pursued up to the second order in~\cite{WG15} to describe~$\langle u\rangle$, where~$u$ satisfies the scalar wave equation~\eqref{PDE}. On taking~$|Y|=1$ for convenience and describing the featured long-wavelength, low-frequency (LW-LF) regime via scalings 
\begin{equation}\label{LW-LF}
\bk=\epsilon\hspace*{0.7pt} \hat{\bk}, \quad {\omega}= \epsilon\hspace*{0.7pt}  \hat{\omega}, \qquad \epsilon=o(1),
\end{equation}
the second-order approximation of the impedance function stemming from the results in~\cite{WG15} can be written as 
\begin{eqnarray} 
\tilde{\mathcal{Z}}^{e}_2(\bk,\omega)  \:~=~\:  
\epsilon^2 \big({\bmu}^{\mbox{\tiny{(0)}}}\!: (i\hat{\bk})^{2} + \rho_0\hspace*{1pt}  \hat{\omega}^2\big) \:+\: 
\epsilon^4 \big({\bmu}^{\mbox{\tiny{(2)}}}\!: (i\hat{\bk})^{4} + {\brho}^{\mbox{\tiny{(2)}}}\!:(i\hat{\bk})^{2}\hat{\omega}^2\big), 
\label{IntroDisComparisonZ2ts}
\end{eqnarray}
where 
\[
(i\hat{\bk})^{n} ~=~ i\hat{\bk}\otimes i\hat{\bk}\ldots\otimes i\hat{\bk} \qquad \text{$n$ times}; 
\]
``:'' denotes $n$-tuple contraction between two $n$th-order tensors producing a scalar; $\rho_0$ is a constant; $\bmu^{\mbox{\tiny{(0)}}}$ and $\brho^{\mbox{\tiny{(0)}}}$ are constant second-order tensors, and $\bmu^{\mbox{\tiny{(2)}}}$ is a constant fourth-order tensor. Later we shall specify these coefficients of homogenization.

\subsection{The main result}

In contrast to~\eqref{IntroDisComparisonZ2ts} whose roots $\tilde{\mathcal{Z}}^{e}_2(\bk,\omega)=0$ \emph{approximate} the acoustic branch in the LW-LF regime, the Willis' effective impedance~$\tilde{Z}^e$ given by~\eqref{EffectiveImpedancebyEffectiveRelations} is capable of capturing the dispersion relationship \emph{exactly} within the~$(\bk,\omega)$ region amenable to homogenization. In this setting one is tempted to obtain a second-order approximation, $\tilde{Z}^e_2$, of~\eqref{EffectiveImpedancebyEffectiveRelations} assuming long wavelengths and low frequencies as in~\eqref{LW-LF}, thus posing a natural question: what is the relationship between $\tilde{\mathcal{Z}}^{e}_2$ and $\tilde{Z}^e_2$? This issue was touched upon in~\cite{NHA16}, inferring the equivalency between the two approximations. In this work, we show for the first time that the two approximations \emph{differ} by a polynomial-type factor, namely 
\begin{eqnarray} \label{IntroDisComparisonWtsZ2}
\tilde{Z}^{e}_2 ~\overset{\epsilon^4}{=}~ \tilde{M}^{-1}_2\tilde{\mathcal{Z}}^{e}_2,
\end{eqnarray}
where $\tilde{M}_2$ is a polynomial in $\bk$ and~$\omega$, while ``$\overset{\epsilon^4}{=}$'' implies equality up to, and including, the $O(\epsilon^4)$ term. As it turns out, equations~$\tilde{\mathcal{Z}}^{e}_2(\bk,\omega)=0$ and~$\tilde{Z}^{e}_2(\bk,\omega)=0$ do provide equivalent approximations of the acoustic branch, $\omega=\tilde{\lambda}_1^{1/2}(\bk)$. However for pairs~$(\bk,\omega)$ \emph{off the acoustic branch}, $\tilde{\mathcal{Z}}^{e}_2$ and $\tilde{Z}^e_2$ differ due to the fact that the two-scale homogenization approach~\cite{PBL78} normally assumes~$f=0$ in~\eqref{PDE}. By reworking the latter analysis with~$f\neq 0$, we show that~$\tilde{M}_2$ arises naturally in the two-scale asymptotic analysis as a \emph{modulation} of the source term, and we establish the corresponding treatment of the dipole source~$\bgamma\neq\bzero$.

\subsection{Asymptotic expansion of the Willis' effective impedance}
In what follows, we establish a formal LW-LF analysis of~$\tilde{Z}^e$. To this end, we consider the asymptotics of $\tilde{w}$ as governed by~\eqref{CellPDEfw} and~\eqref{CellPDEfwBC} since $\langle\tilde{w}\rangle^{-1}=\tilde{Z}^e$. On imposing the LW-LF regime according to~\eqref{LW-LF}, we have 
\begin{eqnarray} 
&&-\epsilon^2 \hat{\omega}^2 \rho (\bx) \tilde{w} -\big( \nabla \!+\! \epsilon\hh i \hat{\bk}\big) \!\cdot\! \big ( G(\bx) ( \nabla \!+\! \epsilon\hh i \hat{\bk})  \tilde{w} \big) ~=~ 1 \quad \mbox{in} \quad Y, \label{ComparisonAnsatzEqn} \\ 
&& \begin{array}{rcl} \tilde{w}|_{x_j=0} &\!\!=\!\!& \tilde{w}|_{x_j=\ell_j}, \\*[1.2mm]
G  (\nabla \!+\! \epsilon i \hat{\bk}) \tilde{w}  \cdot \bnu|_{x_j=0} &\!\!=\!\!& -G (\nabla \!+\! \epsilon i \hat{\bk}) \tilde{w} \cdot \bnu|_{x_j=\ell_j},
\end{array} \quad j=\overline{1,d}.\label{ComparisonAnsatzEqnBCN}\label{ComparisonAnsatzEqnBCD}
\end{eqnarray}
Consider next the asymptotic expansion 
\begin{eqnarray}\label{wexp}
\tilde{w} (\bx) ~=~ \epsilon^{-2} \tilde{w}_0(\bx) \;+\; \epsilon^{-1} \tilde{w}_1(\bx) \;+\; \tilde{w}_2(\bx) \;+\; \epsilon\hh \tilde{w}_3(\bx) \;+\; \epsilon^2 \tilde{w}_4(\bx) \;+\; \cdots,
\end{eqnarray}
by which~(\ref{ComparisonAnsatzEqn})--(\ref{ComparisonAnsatzEqnBCN}) become a series in $\epsilon$. In what follows, the differential equations satisfied by $\tilde{w}_m$ in $Y$ ($m\geqslant 0$) are subject to implicit periodic boundary conditions
\begin{eqnarray*}
\begin{array}{rcl} \tilde{w}_m|_{x_j=0} &\!\!=\!\!& \tilde{w}_m|_{x_j=\ell_j}, \\*[1.2mm]
 G  (\nabla \tilde{w}_{m}+  i \hat{\bk}\hh \tilde{w}_{m-1})  \cdot \bnu|_{x_j=0} &\!\!=\!\!& - G  (\nabla \tilde{w}_{m}+  i \hat{\bk}\hh \tilde{w}_{m-1})  \cdot \bnu|_{x_j=\ell_j}, 
\end{array} \quad j=\overline{1,d}, 
\end{eqnarray*}
where $\tilde{w}_{-1}\equiv 0$. We will conveniently denote by~$\ww_m$ the respective \emph{constants of integration} when solving for~$\tilde{w}_m(\bx)$, $m\geqslant 0$. 

\subsubsection{Leading-order approximation}
The $O(\epsilon^{-2})$ contribution stemming from~\eqref{ComparisonAnsatzEqn} and~\eqref{wexp} reads 
\begin{eqnarray*}
&& -\nabla \!\cdot\! \big(G(\bx) \nabla \tilde{w}_0 \big) ~=~ 0 \quad \mbox{in} \quad Y. 
\end{eqnarray*}
As shown in~\cite{PBL78}, this type of differential equation admits (up to an additive constant) a unique periodic solution, whereby~$\tilde{w}_0(\bx) = \ww_0$. The $O(\epsilon^{-1})$ equation is
\begin{eqnarray*}
&&-\nabla \!\cdot\! \big( G(\bx) \nabla \tilde{w}_1\big) - \nabla \!\cdot\! \big( G(\bx)  i \hat{\bk}\hh \tilde{w}_0 \big)-  i \hat{\bk}\hh \!\cdot\! \big( G(\bx) \nabla \tilde{w}_0  \big) ~=~ 0 \quad \mbox{in} \quad Y,
\end{eqnarray*}
which is solved by~$\tilde{w}_1(\bx) = \bchi^{\mbox{\tiny{(1)}}} (\bx) \cdot i \hat{\bk}\hh w_0 + \ww_1$, where $\bchi^{\mbox{\tiny{(1)}}}\in (H_p^1(Y))^d$ is a \emph{zero-mean} vector satisfying 
\begin{eqnarray} \label{ComparisonEqnchi1}
&&\nabla \cdot \big( G (\nabla \bchi^{\mbox{\tiny{(1)}}} \!+\boldsymbol{I}) \big) ~=~0 \quad \mbox{in} \quad Y, \\
&&  \bnu \cdot G  (\nabla  \bchi^{\mbox{\tiny{(1)}}}\!+\boldsymbol{I} )|_{x_j=0} ~=\: - \bnu \cdot G  (\nabla  \bchi^{\mbox{\tiny{(1)}}}\!+\boldsymbol{I} )|_{x_j=\ell_j} , \quad j=\overline{1,d}. \nonumber
\end{eqnarray}
The $O(1)$ equation reads
\begin{multline}\label{w2}
-\nabla \!\cdot\! \big( G(\bx) \nabla \tilde{w}_2 \big)- \nabla \!\cdot\! \big( G(\bx)  i \hat{\bk}\hh \tilde{w}_1\big)  -  i \hat{\bk} \!\cdot\! \big( G(\bx) \nabla \tilde{w}_1\big)\\
-  i \hat{\bk} \!\cdot\! \big( G(\bx)  i \hat{\bk}\hh \tilde{w}_0 \big) 
- \hat{\omega}^2 \rho(\bx)\hh \tilde{w}_0 ~=~ 1 \quad \mbox{in} \quad Y.  
\end{multline}
Averaging~\eqref{w2} over $Y$ demonstrates that 
\begin{eqnarray} \label{ComparisonSecondOrderHomogenizationZerothOrderPDE}
({\bmu}^{\mbox{\tiny{(0)}}} \!: (i\hat{\bk})^2  + \rho_0 \hat{\omega}^2)\hh  w_0 ~=\: -1,
\end{eqnarray}
where 
\begin{eqnarray} \label{SecondOrderHomogenizationZerothOrderCoefficients}
\rho_0=\langle \rho \rangle, \qquad  {\bmu}^{\mbox{\tiny{(0)}}} = \{ \langle G( \nabla  \bchi^{\mbox{\tiny{(1)}}}  + \boldsymbol{I}  ) \rangle \}.
\end{eqnarray}
Note that in~\eqref{SecondOrderHomogenizationZerothOrderCoefficients} and hereafter, $\{\boldsymbol{\cdot}\}$ denotes tensor averaging over all index permutations; in particular for an $n$th-order tensor~$\boldsymbol{\tau}$, one has 
\begin{equation}\label{symtot}
\{\boldsymbol{\tau}\}_{j_1,j_2,\ldots j_n} ~=~ \frac{1}{n!}\sum_{(l_1,l_2,\ldots l_n)\in P} \boldsymbol{\tau}_{l_1,l_2,\ldots l_n}, \qquad j_1,j_2,\ldots j_n \in\overline{1,d}
\end{equation}
where~$P$ denotes the set of all permutations of~$(j_1,j_2,\ldots j_n)$. Such averaged expression for~${\bmu}^{\mbox{\tiny{(0)}}}$ is due to the structure of~${\bmu}^{\mbox{\tiny{(0)}}} \!: (i\hat{\bk})^2$, which is invariant with respect to the index permutation of~${\bmu}^{\mbox{\tiny{(0)}}}$. For brevity, we will also make use of the partial symmetrization 
\begin{equation}\label{sympart}
\{\boldsymbol{\tau}\}'_{j_1,j_2,\ldots j_n} ~=~ \frac{1}{(n\!-\!1)!}\sum_{(l_2,\ldots l_n)\in Q} \boldsymbol{\tau}_{j_1,l_2,\ldots l_n}, \qquad j_1,j_2,\ldots j_n \in\overline{1,d}
\end{equation}
where~$Q$ denotes the set of all permutations of~$(j_2,j_3,\ldots j_n)$.

\begin{remark}
To ensure that~(\ref{ComparisonSecondOrderHomogenizationZerothOrderPDE}) has a solution, we assume  
\begin{eqnarray}\label{hypo1}
 {\bmu}^{\mbox{\tiny{(0)}}} \!: (i\hat{\bk})^2  + \rho_0 \hat{\omega}^2 ~\neq~ 0.
\end{eqnarray}
\end{remark}
\subsubsection{First-order corrector}
Let $\bchi^{\mbox{\tiny{(2)}}} \in \big(H_p^1(Y)\big)^{d \times d} $ be the unique zero-mean, second-order tensor satisfying 
\begin{eqnarray} \label{Comparisonchi2}
&&\nabla \cdot \big( G\big( \nabla \bchi^{\mbox{\tiny{(2)}}} + \{\boldsymbol{I} \otimes \bchi^{\mbox{\tiny{(1)}}}\}'  \big) \big)  + G(\{\nabla  \bchi^{\mbox{\tiny{(1)}}}\} + \boldsymbol{I}) - \frac{\rho}{\rho_0} \bmu^{\mbox{\tiny{(0)}}}  ~=~0 \quad \mbox{in} \quad Y, \\
&&  \bnu \cdot G\big( \nabla \bchi^{\mbox{\tiny{(2)}}} + \{\boldsymbol{I} \otimes \bchi^{\mbox{\tiny{(1)}}}\}'  \big)|_{x_j=0} ~=\; - \bnu \cdot G\big( \nabla \bchi^{\mbox{\tiny{(2)}}} + \{\boldsymbol{I} \otimes \bchi^{\mbox{\tiny{(1)}}}\}' \big)|_{x_j=\ell_j} , \quad j=\overline{1,d},  \nonumber
\end{eqnarray}
and let $\eta^{\mbox{\tiny{(0)}}} \in H_p^1(Y)$ be the unique zero-mean solution of 
\begin{eqnarray} \label{Comparisoneta0}
&&\nabla \cdot \big(G \nabla \eta^{\mbox{\tiny{(0)}}}\big) ~=~ \frac{\rho-\rho_0}{\rho_0} \quad \mbox{in} \quad Y, \\
&& \bnu \cdot G \nabla  \eta^{\mbox{\tiny{(0)}}}|_{x_j=0} ~=\: - \bnu \cdot G \nabla  \eta^{\mbox{\tiny{(0)}}}|_{x_j=\ell_j} , \quad j=\overline{1,d}. \nonumber
\end{eqnarray}
With such definitions, one can show that~\eqref{w2} is solved by $\,\tilde{w}_2 (\bx) \,=\, \bchi^{\mbox{\tiny{(2)}}} (\bx) : (i \hat{\bk})^{2}\hh  w_0 \,+\, \bchi^{\mbox{\tiny{(1)}}} (\bx) \cdot i \hat{\bk}\hh w_1 \,+\, \eta^{\mbox{\tiny{(0)}}}(\bx) \,+\, w_2$.

Proceeding further with the asymptotic analysis, the $O(\epsilon)$ equation is found as 
\begin{multline}\label{w3}
-\nabla \!\cdot\! \big( G(\bx) \nabla \tilde{w}_3\big)- \nabla \!\cdot\! \big( G(\bx) i \hat{\bk}\hh \tilde{w}_2\big) -  i \hat{\bk} \!\cdot\! \big( G(\bx) \nabla \tilde{w}_2\big) \\
-  i \hat{\bk} \!\cdot\! \big( G(\bx)  i \hat{\bk}\hh \tilde{w}_1\big) - \hat{\omega}^2 \rho(\bx)\hh \tilde{w}_1 ~=~0 \quad \mbox{in} \quad Y.
\end{multline}
Averaging~\eqref{w3} over $Y$ gives the equation for constant~$w_1$ as 
\begin{eqnarray} \label{ComparisonSecondOrderHomogenizationFirstOrderPDE}
\big( \bmu^{\mbox{\tiny{(1)}}}\!: (i \hat{\bk})^{3}  + {\brho}^{\mbox{\tiny{(1)}}} \!\cdot\! i\hat{\bk} ~\hat{\omega}^2 \big)w_0 \,+\, \big( {\bmu}^{\mbox{\tiny{(0)}}}\!: (i \hat{\bk})^{2}+ \rho_0 \hat{\omega}^2 \big) w_1 ~=\: -\langle G\nabla \eta^{\mbox{\tiny{(0)}}}\rangle \!\cdot\! i\hat{\bk},
\end{eqnarray}
where 
\begin{eqnarray} \label{SecondOrderHomogenizationFirstOrderCoefficients}
 {\brho}^{\mbox{\tiny{(1)}}}  = \langle \rho \bchi^{\mbox{\tiny{(1)}}}  \rangle, \qquad   
\bmu^{\mbox{\tiny{(1)}}} \;=\; \big\{\big\langle G( \nabla  {\bchi}^{\mbox{\tiny{(2)}}} + \boldsymbol{I} \!\otimes\! {\bchi}^{\mbox{\tiny{(1)}}} ) \big\rangle \big\}.
\end{eqnarray}

\begin{lemma} \label{ComparisonSecondOrderHomogenizationeta0chi1}
The following identity holds
\begin{eqnarray*}
\langle G \nabla \eta^{\mbox{\tiny{(0)}}} \rangle ~=~ \frac{\langle\rho \bchi^{\mbox{\tiny{(1)}}}\rangle}{\rho_0}, 
\end{eqnarray*}
see Appendix, Section \ref{App} for the proof.
\end{lemma}
\begin{remark} \label{mu1rho1mu0}
On the basis of~(\ref{ComparisonSecondOrderHomogenizationZerothOrderPDE}) and Lemma \ref{ComparisonSecondOrderHomogenizationeta0chi1}, (\ref{ComparisonSecondOrderHomogenizationFirstOrderPDE}) can be recast as
\begin{eqnarray} \label{ComparisonSecondOrderHomogenizationFirstOrderPDEForm1}
\big( \bmu^{\mbox{\tiny{(1)}}} -\frac{1}{\rho_0} \{{\brho}^{\mbox{\tiny{(1)}}}\! \otimes \bmu^{\mbox{\tiny{(0)}}} \} \big) \!: (i \hat{\bk})^{3} w_0 \,+\, \big( {\bmu}^{\mbox{\tiny{(0)}}} \!: (i \hat{\bk})^{2}+ \rho_0 \hat{\omega}^2 \big) w_1 ~=~ 0.
\end{eqnarray}
From the $Y$-average of~\eqref{wexp} and the fact that~$\langle \tilde{w} \rangle$ is real-valued, we have that $w_m=\langle\tilde{w}_m\rangle$, $m\geqslant 0$ are real-valued as well. From~\eqref{ComparisonSecondOrderHomogenizationFirstOrderPDEForm1} and hypothesis~\eqref{hypo1}, on the other hand, $w_1$ must be purely imaginary. This demonstrates that
\begin{eqnarray}\label{w1zero}
w_1=0 \qquad \text{and} \qquad \bmu^{\mbox{\tiny{(1)}}}   -\frac{1}{\rho_0} \{ {\brho}^{\mbox{\tiny{(1)}}} \otimes \bmu^{\mbox{\tiny{(0)}}} \} =0. 
\end{eqnarray}
Here it is noted that: (i) the latter identity can alternatively be established using~(\ref{Comparisonchi2}) and integration by parts, and (ii)~the result $w_1=0$ recovers the previous finding~\cite{WG15, CGM16} that the $O(\epsilon)$ bulk correction of a solution to the the time-harmonic wave equation in periodic media vanishes identically in the mean. 
\end{remark}

\subsubsection{Second-order corrector}
Let $\bchi^{\mbox{\tiny{(3)}}} \in \big( H_p^1(Y)\big)^{d\times d \times d}$ be the unique zero-mean, third-order tensor solving 
\begin{eqnarray} \label{Comparisonchi3}
\hspace*{-7mm} &&\nabla \cdot \big(G\big(\nabla \bchi^{\mbox{\tiny{(3)}}}  + \{\boldsymbol{I} \otimes \bchi^{\mbox{\tiny{(2)}}}\}' \big)\big)  + G\big( \{\nabla \bchi^{\mbox{\tiny{(2)}}}\}  + \{\boldsymbol{I} \otimes \bchi^{\mbox{\tiny{(1)}}}\} \big) \nonumber \\
&&\hspace*{80mm} -\frac{1}{\rho_0}\{\rho \bchi^{\mbox{\tiny{(1)}}} \otimes \bmu^{\mbox{\tiny{(0)}}}\} ~=~ 0 \quad \mbox{in} \quad Y, \qquad \\
\hspace*{-7mm}&&  \bnu \cdot G\big( \nabla \bchi^{\mbox{\tiny{(3)}}} + \{\boldsymbol{I} \otimes \bchi^{\mbox{\tiny{(2)}}}\}'  \big)|_{x_j=0} ~=\: - \bnu \cdot G\big( \nabla \bchi^{\mbox{\tiny{(3)}}} + \{\boldsymbol{I} \otimes \bchi^{\mbox{\tiny{(2)}}}\}'  \big)|_{x_j=\ell_j} , \quad j=\overline{1,d}, \nonumber
\end{eqnarray}
and let $\boldeta^{\mbox{\tiny{(1)}}} \in \big( H_p^1(Y)\big)^{d}$ be the unique zero-mean vector given by 
\begin{eqnarray} \label{Comparisoneta1}
&&\nabla \cdot \big( G (\nabla \boldeta^{\mbox{\tiny{(1)}}} + \boldsymbol{I} \eta^{\mbox{\tiny{(0)}}} ) \big) + G \nabla \eta^{\mbox{\tiny{(0)}}} -\frac{\rho \bchi^{\mbox{\tiny{(1)}}}}{\rho_0} ~=~ 0 \quad \mbox{in} \quad Y, \\
&&  \bnu \cdot G(\nabla \boldeta^{\mbox{\tiny{(1)}}} + \boldsymbol{I} \eta^{\mbox{\tiny{(0)}}} )|_{x_j=0} ~=\: - \bnu \cdot G(\nabla \boldeta^{\mbox{\tiny{(1)}}} + \boldsymbol{I} \eta^{\mbox{\tiny{(0)}}} )|_{x_j=\ell_j} , \quad j=\overline{1,d}. \nonumber
\end{eqnarray}
On the basis of~\eqref{w1zero}--\eqref{Comparisoneta1}, one can show that~\eqref{w3} is solved by $\tilde{w}_3 (\bx) =\bchi^{\mbox{\tiny{(3)}}} (\bx):  (i \hat{\bk})^3 w_0 + \bchi^{\mbox{\tiny{(1)}}}(\bx)   \cdot i\hat{\bk}\hh w_2 
+\boldeta^{\mbox{\tiny{(1)}}} (\bx) \cdot i \hat{\bk} + w_3$. For generality, it is noted that~\eqref{w3} is satisfied even for non-trivial values of~$w_1$ provided that the term $w_1\{\bchi^{\mbox{\tiny{(2)}}} (\bx) : (i \hat{\bk})^2 -\big(\rho_0 \hat{\omega}^2 + \bmu^{\mbox{\tiny{(0)}}}: (i\hat{\bk})^2 \big)\eta^{\mbox{\tiny{(0)}}}(\bx)\}$ is added to~$\tilde{w}_3$. 

To complete the second-order expansion of~$\tilde{Z}^e$, we also need the $O(\epsilon^2)$ contribution to~\eqref{ComparisonAnsatzEqn}, which reads 
\begin{multline*}
-\nabla \!\cdot\! \big( G(\bx) \nabla \tilde{w}_4\big) - \nabla \!\cdot\! \big( G(\bx)  i \hat{\bk}\hh \tilde{w}_3\big) -  i \hat{\bk} \!\cdot\! \big( G(\bx) \nabla \tilde{w}_3\big) \\
-  i\hat{\bk} \!\cdot\! \big( G(\bx) i \hat{\bk}\hh \tilde{w}_2\big) - \hat{\omega}^2 \rho(\bx)\hh \tilde{w}_2 ~=~ 0 \quad \mbox{in} \quad Y.
\end{multline*}
Averaging this result over $Y$ yields the equation for constant~$w_2$ as 
\begin{multline} \label{ComparisonSecondOrderHomogenizationSecondOrderPDE}
\big( {\bmu}^{\mbox{\tiny{(2)}}} \!: (i\hat{\bk})^4 + {\brho}^{\mbox{\tiny{(2)}}} \!: (i\hat{\bk})^2  ~  \hat{\omega}^2  \big) w_0 + \big( {\bmu}^{\mbox{\tiny{(0)}}} \!: (i\hat{\bk})^2 + \rho_0 \hat{\omega}^2 \big) w_2  ~= \\
-\big\langle G( \nabla \boldeta^{\mbox{\tiny{(1)}}} + \boldsymbol{I}  \eta^{\mbox{\tiny{(0)}}}) \big\rangle : (i\hat{\bk})^2 - \langle\rho\eta^{\mbox{\tiny{(0)}}}\rangle \hat{\omega}^2,
\end{multline} 
where 
\begin{eqnarray} \label{SecondOrderHomogenizationSecondOrderCoefficients}
{\brho}^{\mbox{\tiny{(2)}}}  ~=~ \langle \rho {\bchi}^{\mbox{\tiny{(2)}}}  \rangle, \qquad
{\bmu}^{\mbox{\tiny{(2)}}}   ~=~  \big\{\big\langle G( \nabla \bchi^{\mbox{\tiny{(3)}}}  + \boldsymbol{I} \otimes \bchi^{\mbox{\tiny{(2)}}} ) \big\rangle \big\}. 
\end{eqnarray}

\subsubsection{Second-order approximation of~$\tilde{Z}^{e}$}
From the expressions for~$\tilde{w}_j$ $(j=0,1,2)$ and the fact that the unit cell functions $\bchi^{\mbox{\tiny{(1)}}},\bchi^{\mbox{\tiny{(2)}}}$ and $\eta^{\mbox{\tiny{(0)}}}$ each have zero mean, one in particular finds that $\langle\tilde{w}_j\rangle=w_j$. Accordingly the $Y$-average of~\eqref{wexp} yields 
\begin{eqnarray} \label{SecondOrderHomogenizationwh}
\langle \tilde{w}\rangle ~=~ \epsilon^{-2} w_0 + \epsilon^{-1} w_1 + w_2 + O(\epsilon) ~=~ \epsilon^{-2} w_0 + w_2 + O(\epsilon).   
\end{eqnarray}
Recalling~\eqref{EffectiveImpedanceEigen}, one obtains the second-order approximation of the Willis' effective impedance as 
\begin{eqnarray} \label{zeff21}
\tilde{Z}^e_2 ~=~ \frac{ \epsilon^{2}}{w_0 + \epsilon^2\hh w_2}, 
\end{eqnarray}
which is unique up to an~$O(\epsilon^5)$ residual. From~\eqref{ComparisonSecondOrderHomogenizationZerothOrderPDE}, \eqref{ComparisonSecondOrderHomogenizationSecondOrderPDE} and~\eqref{zeff21}, on the other hand, we have 
\begin{multline}\label{ComparisonSecondOrderHomogenizationConstantPDE}
\Big(\epsilon^2 \big( {\bmu}^{\mbox{\tiny{(2)}}} \!: (i\hat{\bk})^4   + {\brho}^{\mbox{\tiny{(2)}}} \!: (i\hat{\bk})^2~ \hat{\omega}^2 \big) + \big( {\bmu}^{\mbox{\tiny{(0)}}} \!: (i\hat{\bk})^2  + \rho_0 \hat{\omega}^2 \big)\Big) (\tilde{Z}^e_2)^{-1}  ~= \\  
 -\epsilon^{-2} \,-\, \big\langle G( \nabla \boldeta^{\mbox{\tiny{(1)}}} + \boldsymbol{I}  \eta^{\mbox{\tiny{(0)}}}) \big\rangle : (i\hat{\bk})^2 - \langle\rho\eta^{\mbox{\tiny{(0)}}}\rangle \hat{\omega}^2 \,+\, O(\epsilon^2).
\end{multline}
On multiplication by $\epsilon^2$, (\ref{ComparisonSecondOrderHomogenizationConstantPDE}) yields the second-order LW-LF approximation of the Willis' effective impedance as 
\begin{eqnarray}\label{zeff22}
\tilde{Z}^{e}_2(\bk,\omega) \:~\overset{\epsilon^{4}}{=}~\: \tilde{M}^{-1}_2 
\Big( \epsilon^2 \big({\bmu}^{\mbox{\tiny{(0)}}}\!: (i\hat{\bk})^{2} + \rho_0\hspace*{1pt}  \hat{\omega}^2\big) \:+\: 
\epsilon^4 \big({\bmu}^{\mbox{\tiny{(2)}}}\!: (i\hat{\bk})^{4} + {\brho}^{\mbox{\tiny{(2)}}}\!: (i\hat{\bk})^{2}\hat{\omega}^2\big)\Big)
\end{eqnarray}
where~``$\overset{\epsilon^{4}}{=}$'' signifies equality up to (and including) the~$O(\epsilon^4)$ term, and~$\tilde{M}_2$ is a polynominal in~$\bk$ and~$\omega$, namely  
\begin{eqnarray} 
\tilde{M}_2(\bk,\omega) ~=\: -1 
\,-\, \epsilon^2 \Big(\big\langle G( \nabla \boldeta^{\mbox{\tiny{(1)}}} + \boldsymbol{I}  \eta^{\mbox{\tiny{(0)}}}) \big\rangle : (i\hat{\bk})^2 + \langle\rho\eta^{\mbox{\tiny{(0)}}}\rangle \hat{\omega}^2 \Big). \label{ComparisonM2}
\end{eqnarray}

\subsection{Comparison between the effective impedances}

A comparison between~\eqref{IntroDisComparisonZ2ts} and~\eqref{zeff22} reveals that the term multiplied by~$\tilde{M}^{-1}_2$ in~\eqref{zeff21} is precisely the second-order approximation, $\tilde{\mathcal{Z}}^{e}_2$, of the effective impedance obtained via two-scale homogenization~\cite{WG15}. 
Accordingly, we arrive at the following theorem.
\begin{theorem} \label{ComparisonTheorem}
The $O(\epsilon^4)$ LW-LF approximation $\tilde{Z}^{e}_2$ of the Willis' effective impedance differs from its $O(\epsilon^4)$ counterpart $\tilde{\mathcal{Z}}^{e}_2$ obtained via two-scale homogenization~\cite{WG15} by factor $\tilde{M}_2$ which is a polynomial in~$\bk$ and~$\omega$; specifically, one has  
\begin{eqnarray} \label{ComparisonWtsZ2}
\tilde{Z}^{e}_2 ~\overset{\epsilon^{4}}{=}~ \tilde{M}^{-1}_2 \tilde{\mathcal{Z}}^{e}_2,
\end{eqnarray}
where $\tilde{\mathcal{Z}}^{e}_2$, $\tilde{Z}^{e}_2$, and~$\tilde{M}_2$ are given respectively by~\eqref{IntroDisComparisonZ2ts}, \eqref{zeff22}, and~\eqref{ComparisonM2}. 
\end{theorem}

\begin{remark}\label{rem8}
Relationship~\eqref{ComparisonWtsZ2} demonstrates that~$\tilde{Z}^{e}_2(\bk,\omega)=0$ if and only if~$\tilde{\mathcal{Z}}^{e}_2(\bk,\omega)=0$, i.e. that~$\tilde{Z}^{e}_2$ and~$\tilde{\mathcal{Z}}^{e}_2$ both recover the dispersive relationship~$\mathcal{D}^e(\bk,\omega)=0$ in the LW-LF regime. One may further note from~\eqref{ComparisonM2} that to the leading order, $\tilde{Z}^{e}_2$ and~$\tilde{\mathcal{Z}}^{e}_2$ carry the opposite sign. This is a reflection of the fact that the two-scale homogenization approach in~\cite{WG15} analyzes the negative of~\eqref{PDE} with $f=0$ and~$\bgamma=\bzero$. 
\end{remark}

\begin{remark}
When the mass density $\rho$ is constant over~$Y$, the coefficient of homogenization~$\eta^{\mbox{\tiny{(0)}}}$ vanishes identically -- which leaves the~$O(\epsilon^2)$ contribution in~\eqref{ComparisonM2} as~$\langle G \nabla \boldeta^{\mbox{\tiny{(1)}}}\rangle:(i\hat{\bk})^2$. On multiplying~(\ref{Comparisoneta1}) by~$\bchi^{\mbox{\tiny{(1)}}}$ and integrating the result by parts, on the other hand, one finds that~$\langle G \nabla \boldeta^{\mbox{\tiny{(1)}}}\rangle =\langle   \bchi^{\mbox{\tiny{(1)}}} \otimes \bchi^{\mbox{\tiny{(1)}}}\rangle$ which reduces $\tilde{M}_2$ to
\begin{eqnarray*}
\tilde{M}_2 ~=\: -1 \hh-\hh \epsilon^2 \langle\bchi^{\mbox{\tiny{(1)}}} \!\otimes \bchi^{\mbox{\tiny{(1)}}} \rangle : (i \hat{\bk})^2 .
\end{eqnarray*}
Recently, the tensor coefficient $\langle\bchi^{\mbox{\tiny{(1)}}} \!\otimes \bchi^{\mbox{\tiny{(1)}}}\rangle $ was obtained in~\cite{ABV16} via two-scale homogenization as a core of the second-order, source-term correction when analyzing the (time domain) wave equation in periodic media with~$\rho=\text{const.}$ and~$f\neq 0$.
\end{remark}

\begin{remark}\label{rem10}
When the pair~$(\bk,\omega)$ does not lie on a dispersion curve i.e.~$\mathcal{D}^e(\bk,\omega)\neq 0$, \eqref{zeff22} and~\eqref{ComparisonM2} demonstrate that $\langle\tilde{u}\rangle=O(\epsilon^{-2}\tilde{f})$ in the LW-LF regime~\eqref{LW-LF} thanks to the fact that $\tilde{Z}^e\langle\tilde{u}\rangle=\tilde{f}$. 
\end{remark}

\begin{remark}
Assume that ~$\mathcal{D}^e(\bk,\omega)\neq 0$, and consider the Bloch wave equation~\eqref{CellPDEBC1} with~$\tilde{f}\neq 0$ and \mbox{$\tilde{\bgamma}=\bzero$}. Then the definition of the effective impedance \eqref{EffectiveImpedanceEigen}, expansion~\eqref{zeff22}--\eqref{ComparisonM2}, relationship~\eqref{ComparisonWtsZ2}, and Remark~\ref{rem10} show that   
\begin{eqnarray}\label{compa1}
\tilde{Z}^e_2\langle\tilde{u}\rangle ~\overset{\epsilon^{2}}{=}~ \tilde{f}, \qquad \tilde{\mathcal{Z}}^e_2\langle\tilde{u}\rangle ~\overset{\epsilon^{2}}{=}~ \tilde{M}_2\hh\tilde{f},
\end{eqnarray}
in the LW-LF regime~\eqref{LW-LF}. The second equality in~\eqref{compa1} in particular shows that the two-scale homogenization analysis~\cite{PBL78}, which normally focuses on the propagation of free waves i.e. postulates~$\tilde{f}=0$, must be appended to properly account for the presence of the source term in the wave equation. This issue was recently addressed in~\cite{ABV16} assuming~$\rho=\text{const.}$, and will be pursued shortly in the general case when $G=G(\bx)$ and~$\rho=\rho(\bx)$, $\bx\in Y$.
\end{remark}

\subsection{PDE interpretation} \label{PDEIntepretationBodyForce}
Theorem~\ref{ComparisonTheorem} covers the time-harmonic wave equation~\eqref{PDE} in~$\mathbb{R}^d$ by considering the Bloch-wave setting~\eqref{CellPDE} and assuming the LW-LF regime where $\omega=\epsilon \hat{\omega}$ and ${\bk}= \epsilon \hat{\bk}$ as $\epsilon\to 0$, while $Y$ remains fixed. This limiting problem can be alternatively cast as a situation where $|Y|^{1/d}=O(\epsilon)$ as $\epsilon\to 0$, while the frequency remain fixed. In the latter limit that is inherent to the two-scale homogenization analysis~\cite{PBL78}, the second equality in~\eqref{compa1} can be translated into the effective second-order approximation of~\eqref{PDE} with~$f(\bx)\neq 0$ and~$\bgamma(\bx)=\bzero$ as
\begin{multline} \label{SpaceTimeSecondOrderPDE}
-\big(\omega^2\rho^{\mbox{\tiny{(0)}}} \langle u\rangle + \bmu^{\mbox{\tiny{(0)}}} \!: \nabla^2 \langle u\rangle\big) \,-\, \epsilon^2 \big(\omega^2\brho^{\mbox{\tiny{(2)}}} \!: \nabla^2 \langle u\rangle + \bmu^{\mbox{\tiny{(2)}}} :  \nabla^4 \langle u\rangle \big)  ~\overset{\epsilon^{2}}{=}~ \\
 \langle f\rangle \,+\, \epsilon^2 \big(\langle G( \nabla \boldeta^{\mbox{\tiny{(1)}}} \!+ \boldsymbol{I} \eta^{\mbox{\tiny{(0)}}})\rangle : \nabla^2 \langle f\rangle + \omega^2\langle \rho \eta^{\mbox{\tiny{(0)}}}\rangle \langle f\rangle\big) \quad\text{in~~}\mathbb{R}^d,   
\end{multline}
where 
\begin{equation}\label{gradn}
\nabla^n g = \nabla\nabla\ldots\nabla g \quad n~\text{times}. 
\end{equation}
Note that~\eqref{SpaceTimeSecondOrderPDE}, formally obtained via replacing~$i\bk$ by~$\nabla$ in the supporting expressions, generalizes the two-scale homogenization result in~\cite{WG15} by allowing for the presence of a non-trivial source term. This claim will be rigorously established in Section~\ref{last}.    


\section{LW-LF contribution due to body eigenstrain} \label{ApproximationsECR}
Motivated by~\eqref{SpaceTimeSecondOrderPDE}, we next seek to expose the second-order approximation of~\eqref{PDE} in~$\mathbb{R}^d$ with \mbox{$f=0$} and~$\bgamma\neq\bzero$ via an LW-LF expansion of the Willis' effective model. To this end, we consider the asymptotics of $\tilde{\bv}$ satisfying~(\ref{CellPDEfv}) due to the fact that its $j$th component, $\tilde{v}_j=\tilde{\bv}\cdot\be_j$, is generated by the eigenstrain $\tilde{\bgamma}=\be_j$. Accordingly, we consider the system 
\begin{eqnarray} 
&&-( \nabla + \epsilon i \hat{\bk}) \cdot \big( G(\bx)( (\nabla + \epsilon i \hat{\bk})   \tilde{v}_j(\bx) -\be_j) \big) -\epsilon^2 \hat{\omega}^2 \rho (\bx) \tilde{v}_j(\bx) ~=~ 0 \quad\text{in}\quad Y, \label{ComparisonAnsatzEqnv}\\
&& \begin{array}{rcl}
\tilde{v}_j|_{x_j=0} &\!\!=\!\!& \tilde{v}_j|_{x_j=\ell_j}, \\*[1.2mm]
G((\nabla + \epsilon i \hat{\bk}) \tilde{v}_j -\be_j)  \cdot \bnu|_{x_j=0}  &\!\!=\!\!& -G((\nabla + \epsilon i \hat{\bk}) \tilde{v}_j -\be_j) \cdot \bnu|_{x_j=\ell_j},
\end{array} \quad j=\overline{1,d}. \qquad \label{ComparisonAnsatzEqnvBCN}
\end{eqnarray}
and expansion 
\begin{eqnarray*}
\tilde{v}_j (\bx) ~=~ \epsilon^{-2} \tilde{v}_{j0}(\bx) + \epsilon^{-1} \tilde{v}_{j1}(\bx) + \tilde{v}_{j2}(\bx) + \epsilon \tilde{v}_{j3}(\bx) + \epsilon^2 \tilde{v}_{j4}(\bx) +  \cdots.
\end{eqnarray*}
Then equation (\ref{ComparisonAnsatzEqnv})--(\ref{ComparisonAnsatzEqnvBCN}) becomes a series in $\epsilon$. In what follows, the differential equations satisfied by $\tilde{v}_{jn}$ in $Y$ are all subject to implicit (periodic) boundary conditions; in particular on setting~$\tilde{v}_{j(-1)}\equiv 0$, one has 
\begin{eqnarray} \label{ComparisonAnsatzEqnBCNvj}
\begin{array}{rcl}
\tilde{v}_{jn}|_{x_j=0} &\!\!=\!\!& \tilde{v}_{jn}|_{x_j=\ell_j}, \label{ComparisonAnsatzEqnBCDvj}\\*[1.2mm]
G(\nabla \tilde{v}_{jn}+ i \hat{\bk} \tilde{v}_{j(n-1)}) \cdot \bnu|_{x_j=0} &\!\!=\!\!& - G(\nabla \tilde{v}_{jn}+  i \hat{\bk} \tilde{v}_{j(n-1)})  \cdot \bnu|_{x_j=\ell_j}, 
\end{array} \quad j=\overline{1,d} 
\end{eqnarray}
for $n\neq 2$, and 
\begin{eqnarray} \label{ComparisonAnsatzEqnBCNvj}
\begin{array}{rcl}
\tilde{v}_{j2}|_{x_j=0} &\!\!=\!\!& \tilde{v}_{j2}|_{x_j=\ell_j}, \label{ComparisonAnsatzEqnBCDvj}\\*[1.2mm]
G(\nabla \tilde{v}_{j2}+ i \hat{\bk} \tilde{v}_{j1}-\be_j) \cdot \bnu|_{x_j=0} &\!\!=\!\!& - G(\nabla \tilde{v}_{j2}+  i \hat{\bk} \tilde{v}_{j1}-\be_j)  \cdot \bnu|_{x_j=\ell_j}, 
\end{array} \quad j=\overline{1,d}. 
\end{eqnarray}
We also denote by~$v_{jn}$ the respective constants of integration when solving for~$\tilde{v}_{jn}$, $n\geqslant 0$. 

\subsection{Leading-order approximation}
The $O(\epsilon^{-2})$ contribution to~\eqref{ComparisonAnsatzEqnv} is
\begin{eqnarray*}
-\nabla \!\cdot\! \big( G(\bx) \nabla \tilde{v}_{j0}\big) ~=~0 \quad \text{in}\quad Y,
\end{eqnarray*}
which yields $\tilde{v}_{j0}(\bx) = v_{j0}$. The $O(\epsilon^{-1})$ equation reads
\begin{eqnarray*}
-\nabla \!\cdot\! \big( G(\bx) \nabla \tilde{v}_{j1}\big) - \nabla \!\cdot\! \big( G(\bx)  i \hat{\bk}\hh \tilde{v}_{j0} \big)-  i \hat{\bk} \!\cdot\! \big( G(\bx) \nabla \tilde{v}_{j0}  \big) ~=~0 \quad \text{in}\quad Y,
\end{eqnarray*}
giving~$\tilde{v}_{j1}(\bx) = \bchi^{\mbox{\tiny{(1)}}} (\bx) \cdot i \hat{\bk}\hh\hh v_{j0} + v_{j1}$, where~$\bchi^{\mbox{\tiny{(1)}}}$ is given by~\eqref{ComparisonEqnchi1}.
The $O(1)$ contribution is
\begin{multline*}
-\nabla \!\cdot\! \big(G(\bx) \nabla \tilde{v}_{j2}\big)- \nabla \!\cdot\! \big(G(\bx) i\hat{\bk}\hh \tilde{v}_{j1}\big) -  i \hat{\bk} \!\cdot\! \big( G(\bx) \nabla \tilde{v}_{j1}\big) \\
-i \hat{\bk} \!\cdot\! \big( G(\bx)  i \hat{\bk}\hh \tilde{v}_{j0} \big) - \hat{\omega}^2 \rho(\bx) \tilde{v}_{j0} ~=\: -\nabla \!\cdot\! (G(\bx) \be_j) \quad \text{in}\quad Y.
\end{multline*}
Averaging the last result over $Y$ yields
\begin{eqnarray*}  
\big( {\bmu}^{\mbox{\tiny{(0)}}} \!: (i\hat{\bk})^2  + \rho_0 \hat{\omega}^2 \big) v_{j0} ~=~ 0,
\end{eqnarray*}
where ${\bmu}^{\mbox{\tiny{(0)}}}$ and $\rho_0$ are given by~\eqref{SecondOrderHomogenizationZerothOrderCoefficients}. Thanks to hypothesis~\eqref{hypo1}, one obtains~$v_{j0}=0$. This reduces the $O(1)$ equation to 
\begin{eqnarray*}
-\nabla \!\cdot\! \big( G(\bx) \nabla \tilde{v}_{j2}(\bx) \big) - \nabla \!\cdot\! \big( G(\bx)  i \hat{\bk}\hh  v_{j1} \big) ~=\: -\nabla \!\cdot\! (G(\bx)\be_j) \quad \text{in}\quad Y,
\end{eqnarray*}
whereby $\,\tilde{v}_{j2}(\bx) = \bchi^{\mbox{\tiny{(1)}}}(\bx) \!\cdot\! i \hat{\bk}\hh  v_{j1} -\bchi^{\mbox{\tiny{(1)}}}(\bx) \!\cdot\! \be_j + \bv_{j2}$. The~$O(\epsilon)$ contribution to~\eqref{ComparisonAnsatzEqnv} reads
\begin{multline*}
-\nabla \!\cdot\! \big(G(\bx) \nabla \tilde{v}_{j3}(\bx) \big)- \nabla \!\cdot\! \big( G(\bx)  i \hat{\bk}\hh \tilde{v}_{j2}(\bx) \big)  -  i \hat{\bk} \!\cdot\!\big(G(\bx)\nabla \tilde{v}_{j2}(\bx) \big)\\
-  i \hat{\bk} \!\cdot\! \big( G(\bx)  i \hat{\bk}\hh \tilde{v}_{j1} \big) 
- \hat{\omega}^2 \rho(\bx) \tilde{v}_{j1} ~=\; -i\hat{\bk}\!\cdot\! (G(\bx) \be_j) \quad \text{in}\quad Y,
\end{multline*}
whose $Y$-average is
\begin{eqnarray} \label{ComparisonSecondOrderHomogenizationZerothOrderPDEv}
\big({\bmu}^{\mbox{\tiny{(0)}}} \!: (i\hat{\bk})^2 + \rho_0 \hat{\omega}^2\big) v_{j1} ~=~ \be_j \!\cdot\! (\bmu^{\mbox{\tiny{(0)}}} \!\cdot\! i \hat{\bk}),
\end{eqnarray}
which makes use of the following lemma.
\begin{lemma} \label{LemmaEigenLeadingOrderSymmetry}
The following identity holds
\begin{eqnarray*}
\langle G\nabla\bchi^{\mbox{\tiny{(1)}}}\rangle ~=~ \langle G\nabla\bchi^{\mbox{\tiny{(1)}}}\rangle\tran,
\end{eqnarray*}
see Appendix, Section~\ref{App} for the proof.
\end{lemma}
\subsection{First-order corrector}
In the sequel we introduce two additional cell functions, $\tilde{\bchi}^{\mbox{\tiny{(2)}}}$ and~$\tilde{\bchi}^{\mbox{\tiny{(3)}}}$, not to be confused with ${\bchi}^{\mbox{\tiny{(2)}}}$ and~${\bchi}^{\mbox{\tiny{(3)}}}$ solving~\eqref{Comparisonchi2} and~\eqref{Comparisonchi3}, respectively. In particular, let~$\tilde{\bchi}^{\mbox{\tiny{(2)}}} \in (H_p^1(Y))^{d\times d}$ be the unique non-symmetric, second-order tensor of zero mean satisfying 
\begin{eqnarray} \label{Comparisonchi2vNonSymmetric}
\hspace{-17mm}&&\nabla \!\cdot\! \big( G  \nabla \tilde{\bchi}^{\mbox{\tiny{(2)}}} \big)  +
\big( \nabla \!\cdot\! \big( G \boldsymbol{I} \otimes \bchi^{\mbox{\tiny{(1)}}}   \big) \big)\tran + 
 G\big(\big(\nabla  \bchi^{\mbox{\tiny{(1)}}}\big)\tran + \boldsymbol{I}\big) - \frac{\rho}{\rho_0} \bmu^{\mbox{\tiny{(0)}}}  ~=~0 \quad \mbox{in} \quad Y, \\
\hspace{-17mm}&&  \big( \bnu \cdot G \nabla \tilde{\bchi}^{\mbox{\tiny{(2)}}}+ \big( \bnu \cdot \big( G \boldsymbol{I} \otimes \bchi^{\mbox{\tiny{(1)}}}   \big) \big)\tran \big)|_{x_j=0}  \nonumber \\
&&\; \hspace{40mm}=~ - \big( \bnu \cdot G \nabla \tilde{\bchi}^{\mbox{\tiny{(2)}}}+ \big( \bnu \cdot \big( G \boldsymbol{I} \otimes \bchi^{\mbox{\tiny{(1)}}}   \big) \big)\tran \big)|_{x_j=\ell_j} , \quad j=\overline{1,d}.  \nonumber
\end{eqnarray}
From~\eqref{Comparisonchi2vNonSymmetric}, one can show that $\tilde{v}_{j3} (\bx) \,=\, \bchi^{\mbox{\tiny{(2)}}} (\bx) : (i \hat{\bk})^2 v_{j1} + \bchi^{\mbox{\tiny{(1)}}}(\bx)  \cdot  i \hat{\bk}\hh v_{j2} -\be_j \cdot(\tilde{\bchi}^{\mbox{\tiny{(2)}}} \!\cdot i \hat{\bk}) + v_{j3}$. With such solution at hand, the $O(\epsilon^2)$ contribution to~\eqref{ComparisonAnsatzEqnv} can be written as 
\begin{multline*}
-\nabla \!\cdot\! \big( G(\bx) \nabla \tilde{v}_{j4}(\bx) \big)- \nabla \!\cdot\! \big( G(\bx)  i \hat{\bk}\hh \tilde{v}_{j3}(\bx) \big) -  i \hat{\bk} \!\cdot\! \big( G(\bx) \nabla \tilde{v}_{j3}(\bx) \big) \\
-  i \hat{\bk} \!\cdot\! \big( G(\bx)  i \hat{\bk}\hh \tilde{v}_{j2}(\bx) \big) - \hat{\omega}^2 \rho(\bx) \tilde{v}_{j2}(\bx)=0 \quad \text{in} \quad Y.
\end{multline*}
Averaging this result over $Y$ gives the algebraic equation for~$\tilde{v}_{j2}$ as
\begin{multline} \label{ComparisonSecondOrderHomogenizationFirstOrderPDEv}
\big( {\bmu}^{\mbox{\tiny{(1)}}} \!: (i\hat{\bk})^3   +  {\brho}^{\mbox{\tiny{(1)}}} \!\cdot i\hat{\bk}\hh \hat{\omega}^2 \big) v_{j1} + \big( {\bmu}^{\mbox{\tiny{(0)}}} \!: (i\hat{\bk})^2  + \rho_0 \hat{\omega}^2 \big) v_{j2} ~=~
\tilde{\bmu}^{\mbox{\tiny{(1)}}}: \big(\be_j \otimes (i\hat{\bk})^2 \big) + \hat{\omega}^2 {\brho}^{\mbox{\tiny{(1)}}} \!\cdot \be_j,
\end{multline}
where $\brho^{\mbox{\tiny{(1)}}}$ is defined in~(\ref{SecondOrderHomogenizationFirstOrderCoefficients}), and  $\tilde{\bmu}^{\mbox{\tiny{(1)}}}$ is a third-order tensor given by
\begin{eqnarray*}
\tilde{\bmu}^{\mbox{\tiny{(1)}}} ~=~ \big\{\big\langle G \big((\nabla \tilde{\bchi}^{\mbox{\tiny{(2)}}})\tran +\bchi^{\mbox{\tiny{(1)}}}\!\otimes \boldsymbol{I} \big)\big \rangle \big\}'.
\end{eqnarray*}
Here the partial symmetrization operator~$\{\cdot\}'$ is given by~\eqref{sympart}, and the transpose of a third order tensor is defined as~$(\boldsymbol{\tau})\tran_{k\ell m} = (\boldsymbol{\tau})_{\ell k m}$.
\subsection{Second-order corrector}
Let $\tilde{\bchi}^{\mbox{\tiny{(3)}}} \in (H_p^1(Y))^{d\times d \times d}$ be the unique non-symmetric, third-order tensor of zero mean satisfying
\begin{eqnarray} \label{Comparisonchi3vNonsymmetric}
\hspace*{-10mm}&&  \nabla \!\cdot\! \big( G  \nabla \tilde{\bchi}^{\mbox{\tiny{(3)}}} \big) +\big( \nabla \!\cdot\! \big( G \boldsymbol{I} \otimes \bchi^{\mbox{\tiny{(2)}}}   \big) \big)\tran  + G\big(  \bchi^{\mbox{\tiny{(1)}}} \otimes \boldsymbol{I} + \big\{\big(\nabla \tilde{\bchi}^{\mbox{\tiny{(2)}}} \big)\tran \big\}' \big) \nonumber \\
\hspace*{-10mm}&& \hspace*{57mm} -\; \tilde{\bmu}^{\mbox{\tiny{(1)}}}  - \big\{\bmu^{\mbox{\tiny{(0)}}} \otimes \frac{\rho \bchi^{\mbox{\tiny{(1)}}} \!-\! \brho^{\mbox{\tiny{(1)}}} }{\rho_0} \big\}' ~=~ 0 \quad \text{in} \quad Y, \\
\hspace*{-10mm}&& \big( \bnu \cdot G \nabla \tilde{\bchi}^{\mbox{\tiny{(3)}}}+ \big( \bnu \cdot \big( G \boldsymbol{I} \otimes \bchi^{\mbox{\tiny{(2)}}}   \big) \big)\tran \big)|_{x_j=0} \nonumber \\
&&\hspace*{33mm}~=\; - \big( \bnu \cdot G \nabla \tilde{\bchi}^{\mbox{\tiny{(3)}}}+ \big( \bnu \cdot \big( G \boldsymbol{I} \otimes \bchi^{\mbox{\tiny{(2)}}}   \big) \big)\tran \big)|_{x_j=\ell_j} , \quad j=\overline{1,d}, \nonumber 
\end{eqnarray}
Further, let~$\balpha^{\mbox{\tiny{(1)}}} \in (H_p^1(Y))^{d}$ be the unique vector of average $0$ that satisfies
\begin{eqnarray} \label{Comparisonalpha1}
\nabla \!\cdot\! \big( G \nabla \balpha^{\mbox{\tiny{(1)}}} \big) &\!\!\!=\!\!\!& \rho \bchi^{\mbox{\tiny{(1)}}} - \brho^{\mbox{\tiny{(1)}}} \quad \text{in} \quad Y, \\
\bnu \cdot G \nabla  \balpha^{\mbox{\tiny{(1)}}}|_{x_j=0} &\!\!\!=\!\!\!& - \bnu \cdot G \nabla  \balpha^{\mbox{\tiny{(1)}}}|_{x_j=\ell_j} , \quad j=\overline{1,d}. \nonumber
\end{eqnarray}
Then one can show that $\tilde{v}_{j4} (\bx) =\bchi^{\mbox{\tiny{(3)}}} (\bx) : (i \hat{\bk})^3  v_{j1}+ \bchi^{\mbox{\tiny{(2)}}} (\bx) : (i \hat{\bk})^2 v_{j2} + \bchi^{\mbox{\tiny{(1)}}}(\bx) \cdot  i \hat{\bk}\hh v_{j3} 
-(\rho_0 \hat{\omega}^2 + \bmu^{\mbox{\tiny{(0)}}} \!: (i\hat{\bk})^2 )\eta^{\mbox{\tiny{(0)}}} v_{j2} + \hat{\omega}^2 \balpha^{\mbox{\tiny{(1)}}} \!\cdot \be_j  - \tilde{\bchi}^{\mbox{\tiny{(3)}}}(\bx) :  ( \be_j \otimes i \hat{\bk} \otimes i \hat{\bk}) + v_{j4}$, where $\bchi^{\mbox{\tiny{(3)}}}$ is given by~\eqref{Comparisonchi3}.

To complete the analysis, we also need the~$O(\epsilon^3)$ equation which reads 
\begin{multline*}
-\nabla \!\cdot\! \big( G(\bx) \nabla \tilde{v}_{j5}(\bx) \big) - \nabla \!\cdot\! \big( G(\bx)  i \hat{\bk}\hh \tilde{v}_{j4}(\bx) \big) -  i \hat{\bk} \!\cdot\! \big( G(\bx) \nabla \tilde{v}_{j4}(\bx) \big) \\
-  i \hat{\bk} \!\cdot\! \big( G(\bx)  i \hat{\bk}\hh \tilde{v}_{j3}(\bx) \big) - \hat{\omega}^2 \rho(\bx) \tilde{v}_{j3}(\bx) ~=~0 \quad \text{in} \quad Y.
\end{multline*}
Averaging this result over $Y$ yields the equation for~$v_{j3}$ as
\begin{eqnarray} \label{ComparisonSecondOrderHomogenizationSecondOrderPDEv}
\hspace*{-10mm}&&({\bmu}^{\mbox{\tiny{(2)}}} \!: (i\hat{\bk})^4  + {\brho}^{\mbox{\tiny{(2)}}} \!: (i\hat{\bk})^2   \hat{\omega}^2 )v_{j1} + 
\big({\bmu}^{\mbox{\tiny{(1)}}} \!: (i\hat{\bk})^3   + {\brho}^{\mbox{\tiny{(1)}}} \cdot i\hat{\bk}\hh \hat{\omega}^2 -
({\bmu}^{\mbox{\tiny{(0)}}} \!: (i\hat{\bk})^2  +\rho_0 \hat{\omega}^2 ) \rho_0^{-1} \brho^{\mbox{\tiny{(1)}}} \!\cdot i\hat{\bk}\big)v_{j2} +  \nonumber \\
\hspace*{-10mm}&& 
( {\bmu}^{\mbox{\tiny{(0)}}}: (i\hat{\bk})^2   +\rho_0 \hat{\omega}^2 )\hh v_{j3} ~=~
\tilde{\bmu}^{\mbox{\tiny{(2)}}} \!: \big(\be_j \!\otimes (i\hat{\bk})^3 \big)  + \tilde{\brho}^{\mbox{\tiny{(2)}}}:(\be_j \!\otimes  i\hat{\bk} )  \hat{\omega}^2 -\langle \rho \bchi^{\mbox{\tiny{(1)}}} \!\otimes \bchi^{\mbox{\tiny{(1)}}} \rangle : (\be_j \!\otimes i \hat{\bk}) \hat{\omega}^2, 
\end{eqnarray}
where ${\bmu}^{\mbox{\tiny{(2)}}}$ and $\brho^{\mbox{\tiny{(2)}}}$ are given by~(\ref{SecondOrderHomogenizationSecondOrderCoefficients}); 
\begin{eqnarray*}
\tilde{\bmu}^{\mbox{\tiny{(2)}}} ~=~ \big\{\big\langle G \big((\nabla\tilde{\bchi}^{\mbox{\tiny{(3)}}})\tran + \tilde{\bchi}^{\mbox{\tiny{(2)}}} \!\otimes \boldsymbol{I} \big) \big\rangle \big\}', \qquad 
\tilde{\brho}^{\mbox{\tiny{(2)}}} ~=~ \langle \rho  \tilde{\bchi}^{\mbox{\tiny{(2)}}} \rangle, 
\end{eqnarray*}
and the transpose of a fourth order-tensor is defined as $\big( \boldsymbol{\tau}\big)\tran_{k\ell m n} =  \boldsymbol{\tau}_{\ell k m n}$. Note that in~(\ref{ComparisonSecondOrderHomogenizationSecondOrderPDEv}), the term $\langle \rho \bchi^{\mbox{\tiny{(1)}}} \!\otimes \bchi^{\mbox{\tiny{(1)}}} \rangle$ is due to $\langle G \nabla \balpha^{\mbox{\tiny{(1)}}} \rangle$ and the following lemma.
\begin{lemma} \label{ComparisonSecondOrderHomogenizationAlpha1}
The following identity holds
$$
\langle G \nabla \balpha^{\mbox{\tiny{(1)}}} \rangle ~=~ \langle \rho \bchi^{\mbox{\tiny{(1)}}} \!\otimes \bchi^{\mbox{\tiny{(1)}}} \rangle.
$$
see Appendix, Section~\ref{App} for the proof.
\end{lemma}
\subsection{Second-order approximation}
From the above results, we have 
\begin{eqnarray} \label{SecondOrderHomogenizationvh}
\langle \tilde{v}_j \rangle ~=~ \epsilon^{-1} v_{j1} +  v_{j2}  + \epsilon\hh v_{j3}+ O(\epsilon^2). 
\end{eqnarray}
By virtue of~(\ref{ComparisonSecondOrderHomogenizationZerothOrderPDEv}), (\ref{ComparisonSecondOrderHomogenizationFirstOrderPDEv}) and~(\ref{ComparisonSecondOrderHomogenizationSecondOrderPDEv}), on the other hand, one can show that 
\begin{multline} 
\epsilon^2({\bmu}^{\mbox{\tiny{(2)}}}: (i\hat{\bk})^4 + {\brho}^{\mbox{\tiny{(2)}}} \!: (i\hat{\bk})^2    \hat{\omega}^2 ) \langle\tilde{v}_j\rangle  + ( {\bmu}^{\mbox{\tiny{(0)}}} \!: (i\hat{\bk})^2   +\rho_0 \hat{\omega}^2 ) \langle\tilde{v}_j\rangle ~= \\
 \epsilon^{-1} \bmu^{\mbox{\tiny{(0)}}} \!: (\be_j \!\otimes i \hat{\bk})- \big\{\bmu^{\mbox{\tiny{(0)}}} \!\otimes \frac{\brho^{\mbox{\tiny{(1)}}}}{\rho_0}\big\}' : \big( \be_j \!\otimes (i\hat{\bk})^2  \big) + \tilde{\bmu}^{\mbox{\tiny{(1)}}} \!: \big( \be_j \!\otimes (i\hat{\bk})^2  \big)  + \hat{\omega}^2 {\brho}^{\mbox{\tiny{(1)}}} \!\cdot \be_j  \nonumber \\
+ \epsilon \tilde{\bmu}^{\mbox{\tiny{(2)}}} \!: \big(\be_j \!\otimes (i\hat{\bk})^3 \big) + \epsilon  \tilde{\brho}^{\mbox{\tiny{(2)}}}\!:(\be_j \!\otimes  i\hat{\bk} )  \hat{\omega}^2 -\epsilon \hat{\omega}^2\langle \rho \bchi^{\mbox{\tiny{(1)}}} \!\otimes \bchi^{\mbox{\tiny{(1)}}} \rangle : (\be_j \!\otimes i \hat{\bk}) \;+ O(\epsilon^2).
\end{multline}
On multiplying the last result by~$\epsilon^2$, we obtain 
\begin{eqnarray} \label{ComparisonSecondOrderHomogenizationConstantPDE1v}
\tilde{\mathcal{Z}}^{e}_2  \langle\tilde{v}_j\rangle ~=~ \tilde{N}_2 \;+\; O(\epsilon^4), 
\end{eqnarray}
where~$\tilde{\mathcal{Z}}^{e}_2$ is the two-scale impedance function given by~\eqref{IntroDisComparisonZ2ts}, and 
\begin{eqnarray*}
\tilde{N}_2 &=&  \epsilon \bmu^{\mbox{\tiny{(0)}}} \!: (\be_j \!\otimes i \hat{\bk})-\epsilon^{2} \big\{\bmu^{\mbox{\tiny{(0)}}} \!\otimes \frac{\brho^{\mbox{\tiny{(1)}}}}{\rho_0} \big\}' : \big( \be_j \!\otimes (i\hat{\bk})^2  \big) + \epsilon^{2} \tilde{\bmu}^{\mbox{\tiny{(1)}}} \!: \big( \be_j \!\otimes (i\hat{\bk})^2  \big)  +\epsilon^{2} \hat{\omega}^2 {\brho}^{\mbox{\tiny{(1)}}} \!\cdot \be_j  \nonumber \\
&& ~+ \epsilon^3 \tilde{\bmu}^{\mbox{\tiny{(2)}}} \!: \big(\be_j \!\otimes (i\hat{\bk})^3 \big) + \epsilon^3  \tilde{\brho}^{\mbox{\tiny{(2)}}}\!:(\be_j \!\otimes  i\hat{\bk} )  \hat{\omega}^2 -\epsilon^3 \hat{\omega}^2\langle \rho \bchi^{\mbox{\tiny{(1)}}} \!\otimes \bchi^{\mbox{\tiny{(1)}}} \rangle : (\be_j \!\otimes i \hat{\bk}).
\end{eqnarray*}
Here it is noted that~(\ref{SecondOrderHomogenizationvh}) can also be used to derive the second-order approximation of the Willis' effective constitutive relationship.

\subsection{PDE interpretation}

Following the arguments in Section~\ref{ApproximationsECR}\ref{PDEIntepretationBodyForce}, an effective second-order approximation of the time-harmonic wave equation~\eqref{PDE} with~$f(\bx)=0$ and~$\bgamma(\bx)\neq\bzero$ can now be written as 
\begin{multline} \label{SpaceTimeSecondOrderPDEgam}
-\big(\omega^2 \rho^{\mbox{\tiny{(0)}}} \langle u\rangle + \bmu^{\mbox{\tiny{(0)}}} \!:  \nabla^2 \langle u\rangle \big) - \epsilon^2 \big(\omega^2 \brho^{\mbox{\tiny{(2)}}} \!: \nabla^2 \langle u\rangle + \bmu^{\mbox{\tiny{(2)}}} \!:  \nabla^4 \langle u\rangle \big) ~=~ \\
   -\bmu^{\mbox{\tiny{(0)}}} \!: (\bgam \nabla) \,+\, \epsilon \big(\big\{\bmu^{\mbox{\tiny{(0)}}} \!\otimes \frac{\brho^{\mbox{\tiny{(1)}}}}{\rho_0} \big\}' : (\bgam \nabla^2) - \tilde{\bmu}^{\mbox{\tiny{(1)}}} \!: \big( \bgam \nabla^2  \big) - \hh \omega^2 {\brho}^{\mbox{\tiny{(1)}}} \!\cdot \bgam \big) \\
 -  \epsilon^2 \big(\tilde{\bmu}^{\mbox{\tiny{(2)}}} \!: \big(\bgam \nabla^3  \big) + \omega^2 \tilde{\brho}^{\mbox{\tiny{(2)}}}\!:(\bgam\nabla)    
- \omega^2 \langle \rho \bchi^{\mbox{\tiny{(1)}}} \otimes \bchi^{\mbox{\tiny{(1)}}} \rangle : (\bgam\nabla)\big) \quad \text{in} \quad \mathbb{R}^d,
\end{multline}
where $\bgam\nabla$ denotes gradient \emph{to the left} i.e. $(\bgam \nabla) = \partial\bgam/\partial x_j \otimes \be_j$, and $\bgamma\nabla^n$ is defined by analogy to~\eqref{gradn}. Here it is interesting to note that, in contrast to~\eqref{SpaceTimeSecondOrderPDE}, the second-order approximation~\eqref{SpaceTimeSecondOrderPDEgam} also includes an~$O(\epsilon)$ correction. 

\section{Generalization of the two-scale homogenization approach}\label{last}

In what follows, we demonstrate how the two-scale homogenization approach~\cite{PBL78} can be generalized to handle~\eqref{PDE} with a non-trivial source term ($f\neq 0$, $\bgamma=\bzero$), thus recovering the second-order effective equation~\eqref{SpaceTimeSecondOrderPDE} governing the mean motion $\langle u \rangle$ in~$\mathbb{R}^d$. As examined earlier, we adopt the standard premise of the two-scale analysis that $|Y|^{1/d}=O(\epsilon)$ as $\epsilon\to 0$, while the frequency remain fixed. Specifically, we consider the time-harmonic wave equation 
\begin{eqnarray}\label{newpde}
-\omega^2 \rho(\frac{\bx}{\epsilon})  u \;-\; \nabla \cdot \big( G(\frac{\bx}{\epsilon}) \nabla u  \big) ~=~ f \quad \text{in} \quad \mathbb{R}^d,
\end{eqnarray}
where $G$ and $\rho$ are $Y$-periodic, $\omega=O(1)$, and~$\epsilon=o(1)$. We seek the solution in the form 
\begin{eqnarray}\label{newsol}
 u(\bx) \quad \to \quad u(\bx,\by)~=~\sum_{n=0}^\infty \epsilon^n u_n(\bx,\by), \qquad \by ~=~ \frac{\bx}{\epsilon},
\end{eqnarray}
where $\by$ is the so-called ``fast'' variable describing the variations due to periodic microstructure such that $\langle u(\bx)\rangle = \langle u(\bx,\by)\rangle_{\by}$. On substituting~\eqref{newsol} into~\eqref{newpde}, one obtains 
\begin{eqnarray}\label{newpde2}
-\omega^2 \rho(\by)\hh u - \big(\nabla_{\!\bx}+ \frac{1}{\epsilon} \nabla_{\!\by}\big) \big(G(\by) (\nabla_{\!\bx}+ \frac{1}{\epsilon} \nabla_{\!\by}) u \big) ~=~ f(\bx,\by),
\end{eqnarray}
where we allow for~$f$ to have small-scale fluctuations due to microstructure (e.g. acoustic radiation force generated by high-intensity ultrasound field acting upon a periodic medium). Since~\eqref{newpde2} is now a series in~$\epsilon$, by the hierarchy of equations we find that the second-order solution in~\cite{WG15} can be generalized as 
\begin{eqnarray*}
u_0(\bx,\by) &\!\!\!\!=\!\!\!& U_0(\bx), \\
u_1(\bx,\by) &\!\!\!\!=\!\!\!& U_1(\bx) \:+\: \bchi^{\mbox{\tiny{(1)}}}(\by) \cdot\! \nabla_{\!\bx} U_0(\bx), \\
u_2(\bx,\by) &\!\!\!\!=\!\!\!& U_2(\bx) \:+\:\bchi^{\mbox{\tiny{(2)}}}(\by) :\!\nabla_{\!\bx}^2 U_0(\bx)\:+\:\bchi^{\mbox{\tiny{(1)}}}(\by) \cdot\! \nabla_{\!\bx} U_1(\bx)+ \eta^{\mbox{\tiny{(0)}}}(\by) \langle f\rangle_{\by}(\bx), \\
u_3(\bx,\by) &\!\!\!\!=\!\!\!& U_3(\bx)  +  \bchi^{\mbox{\tiny{(3)}}}(\by) :\!\nabla_{\!\bx}^3 U_0(\bx)  +  \bchi^{\mbox{\tiny{(2)}}}(\by) :\!\nabla_{\!\bx}^2 U_1(\bx)  + \bchi^{\mbox{\tiny{(1)}}}(\by)\!\cdot\! \nabla_{\!\bx} U_2(\bx) \\
&&+ ~\boldeta^{\mbox{\tiny{(1)}}}(\by) \cdot\! \nabla_{\!\bx} \langle f\rangle_{\by}(\bx), 
\end{eqnarray*}
to account for the source term~$f$, where~$U_j$ ($j\!=\!0,1,2$) solve the cascade of differential equations 
\begin{eqnarray} \label{newpde3}
\rho^{\mbox{\tiny{(0)}}} \omega^2 U_0 + \bmu^{\mbox{\tiny{(0)}}} \!: \nabla_{\!\bx}^2 U_0 &\!\!\!=\!\!\!& -\langle f\rangle_{\by}, \\
\rho^{\mbox{\tiny{(0)}}} \omega^2 U_1 + \bmu^{\mbox{\tiny{(0)}}} \!: \nabla_{\!\bx}^2 U_1 &\!\!\!=\!\!\!& 0, \label{newpde3b}\\
\rho^{\mbox{\tiny{(0)}}} \omega^2 U_2 + \bmu^{\mbox{\tiny{(0)}}} \!: \nabla_{\!\bx}^2 U_2 + \rho^{\mbox{\tiny{(2)}}} \omega^2: \nabla_{\!\bx}^2 U_0 + \bmu^{\mbox{\tiny{(2)}}} \!: \nabla_{\!\bx}^4 U_0 &\!\!\!=\!\!\!& - \langle G (\nabla \boldeta^{\mbox{\tiny{(1)}}} \!+ \boldsymbol{I} \eta^{\mbox{\tiny{(0)}}})\rangle : \nabla_{\!\bx}^2 \langle f\rangle_{\by} \nonumber \\
&&-\omega^2 \langle \rho \eta^{\mbox{\tiny{(0)}}} \rangle \langle f\rangle_{\by}  . \label{newpde4}
\end{eqnarray}
On recalling that $\langle f(\bx,\by)\rangle_{\by} = \langle f(\bx)\rangle$ by definition and considering the second-order, mean-field approximation
\[
\langle u(\bx) \rangle ~=~ \sum_{n=0}^{2} \epsilon^{n} \langle u_n(\bx,\by) \rangle_{\by} \,+\, o(\epsilon^2) \,~=~  \sum_{n=0}^{2} \epsilon^{n} U_n(\bx) \,+\, o(\epsilon^2),
\]
one immediately recovers~\eqref{SpaceTimeSecondOrderPDE} by the weighted summation of~\eqref{newpde3}--\eqref{newpde4}.
\begin{remark}
As can be seen from~\eqref{newpde3b}, the homogenized source term has no~$O(\epsilon)$ contribution, which is consistent with the previous two-scale analysis of free waves~\cite{CGM16} stating that the first-order correction~$u_1(\bx,\by$) vanishes in the mean, i.e. $U_1(\bx)=0$. 
\end{remark}
\begin{remark} The two-scale homogenization road to~\eqref{SpaceTimeSecondOrderPDE} can be understood as a two-stage paradigm, where the solution is (i) first expanded in~$\epsilon$ according to~\eqref{newsol}, and then (ii) averaged to arrive at the hierarchical mean-field equations~\eqref{newpde3}--\eqref{newpde4}. In contrast, by adopting the Willis' approach we first average the wavefield solving~\eqref{CellPDE} via an effective constitutive description~\eqref{IntroEffectiveWillis}, and then expand the obtained mean solution in powers of~$\epsilon$. It is perhaps remarkable that, at least under the hypotheses made in this work, the operations of asymptotic expansion and averaging commute when deriving~\eqref{SpaceTimeSecondOrderPDE}. In this vein, using the Willis' approach to obtain the second-order LW-LF approximation can also be thought of as a ``single-scale'' homogenization framework. 
\end{remark}

\section{Numerical example}\label{ImpedanceNumerics}

In this section, we illustrate by a simple example the second-order asymptotics of the Willis' effective impedance, and compare this approximation with its counterpart derived via two-scale homogenization. In particular, we consider the \emph{one-dimensional} periodic structure where the unit cell $Y=(0,1)$ is composed of two homogeneous phases:
\begin{eqnarray}\label{1dmodel}
\rho(x)=1, & G(x)=1 \quad &\mbox{for} \quad 0<x<\tfrac{1}{2}, \\
\rho(x)= \gamma_\rho, & G(x)= \gamma_{\mbox{\tiny{$G$}}} \quad &\mbox{for} \quad \tfrac{1}{2}<x<1. 
\end{eqnarray}
The exact dispersion relationship for this periodic structure is computed using the \emph{bvp4c} function in Matlab. The constants of homogenization $\rho^{\mbox{\tiny{($\ell$)}}}$ and $\mu^{\mbox{\tiny{($\ell$)}}}$ ($\ell=0,1,2$), which are independent of~$k$ and~$\omega$, are computed using FreeFem++ \cite{H13} and Matlab.

As can be seen from Fig.~\ref{0101}, the dispersion curve in the~$(k,\omega)$ space stemming from the second-order model $\tilde{Z}^e_2(k,\omega)=0\,$ provides markedly better LW-LF approximation of the exact relationship than the quasi-static model $\omega=k \sqrt{\mu^{\mbox{\tiny{(0)}}}/\rho^{\mbox{\tiny{(0)}}}}$. In particular, $\tilde{Z}^e_2$ is deemed to furnish a satisfactory approximation up to $k\simeq 2$, which covers more than one half of the first Brillouin zone $k\in[0,\pi)$. For completeness, Fig.~\ref{0601} plots the modulation polynomial~$\tilde{M}_2(k,\omega)$ as given by~\eqref{ComparisonM2} over the region $[0,2\pi)\times[0,2\pi)$. It is noted that for $0\leqslant k\lesssim 2$ where the second-order approximation applies according to Fig.~\ref{0101}, the magnitude of~$\tilde{M}_2$ may drop down to less than 60\% of its quasi-static value~$|\tilde{M}_2|=1$, thus highlighting the necessity to modulate the source term as in~\eqref{compa1} or equivalently~\eqref{SpaceTimeSecondOrderPDE} when using the multiple-scales homogenization approach to study waves due to body forces in periodic media. 

The main result of this work, given by Theorem~\ref{ComparisonTheorem}, is illustrated in Fig.~\ref{WillisTwoScale} which compares the Willis' effective impedance $\tilde{Z}^e_2 (k,\omega)$ with its two-scale counterpart~$\tilde{\mathcal{Z}}^e_2(k,\omega)$ over the region $[0,2\pi)\times[0,2\pi)$. As can be seen from the display, the two second-order approximations of the effective impedance \emph{share} the zero-level set, i.e. $\,\tilde{Z}^e_2 (k,\omega)=0 \:\Leftrightarrow\:  \tilde{\mathcal{Z}}^e_2(k,\omega)=0\,$ (as given by the second contour line from the bottom), see also Remark~\ref{rem8}. Away from the dispersion curve, however, the two approximations exhibit notable differences, especially for the bi-laminate with~$\gamma_\rho=\gamma_{\mbox{\tiny{$G$}}}=0.1$. 

\begin{figure}[hb!] 
\centering{\includegraphics[height=2.0in]{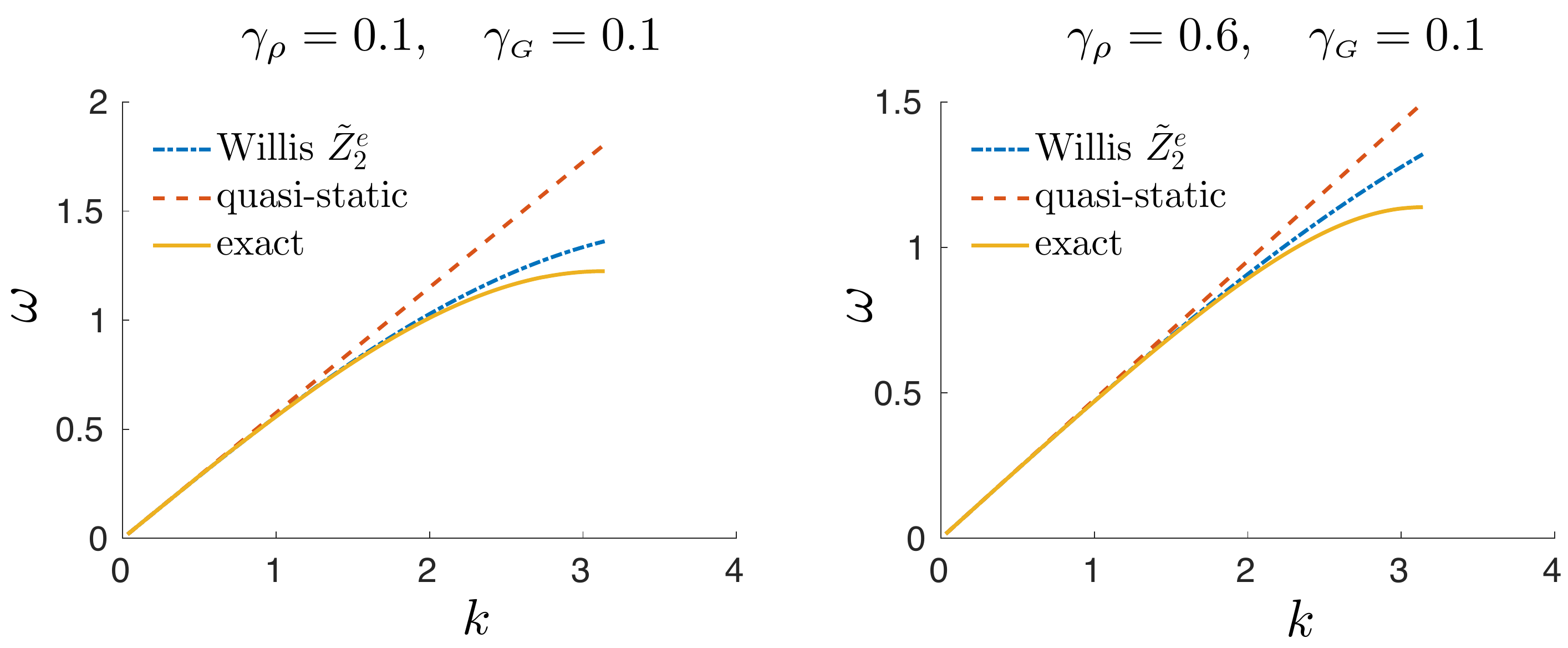}}\vspace*{-2mm}
\caption{First branch of the dispersion relationship for the bi-laminate periodic structure~\eqref{1dmodel} for two example values of~$(\gamma_{\rho},\gamma_{\mbox{\tiny{$G$}}})$: exact solution (solid line) versus second-order LW-LF approximation of the Willis' effective model (dot-dashed line) and the reference quasi-static approximation.} \label{0101}
\end{figure}

\begin{figure}[ht!] 
\centering{\includegraphics[height=2.0in]{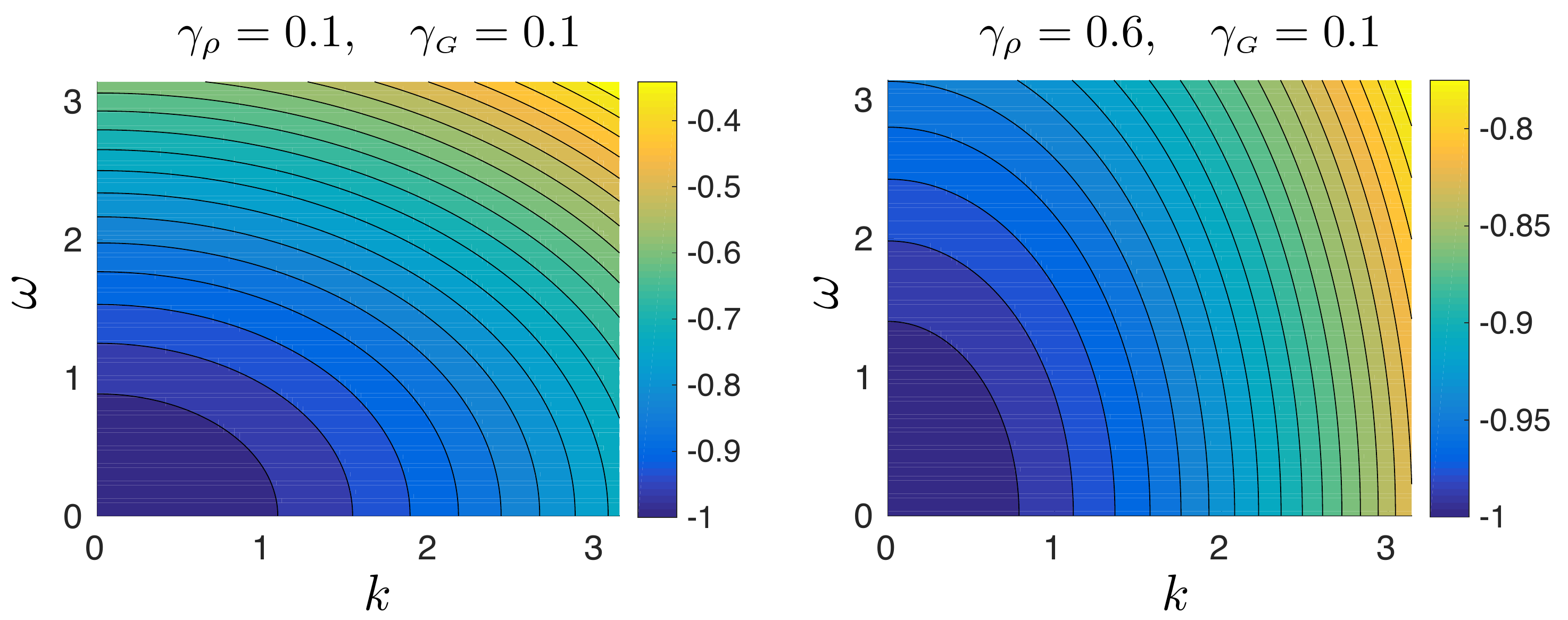}}\vspace*{-2mm}
\caption{Contour map of the modulation polynomial~$\tilde{M}_2(k,\omega)$ for the bi-laminate periodic structure~\eqref{1dmodel}. }\label{0601}
\end{figure}

\begin{figure}[h!] 
\centering{\includegraphics[height=4.0in]{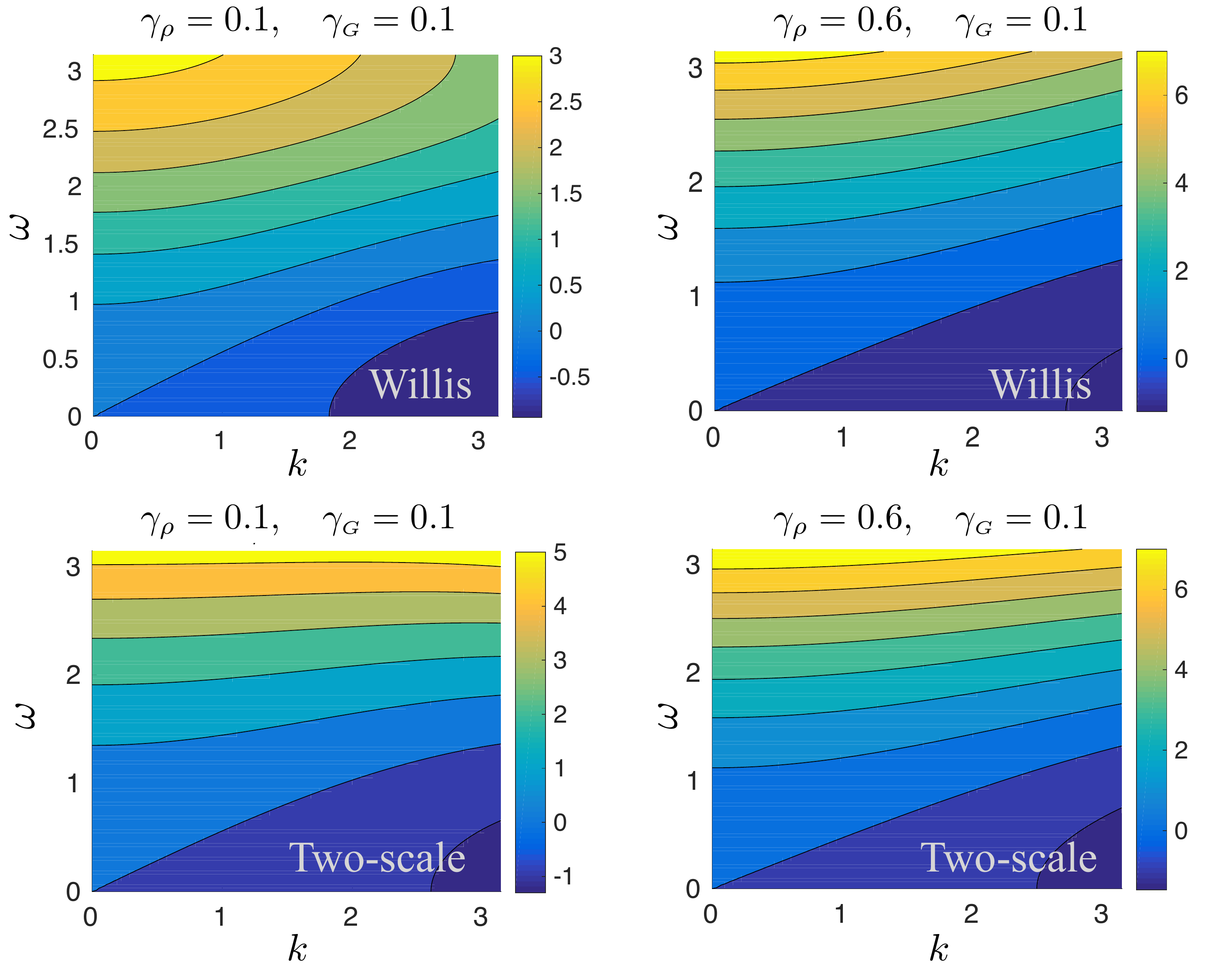}}\vspace*{-2mm}
\caption{Contour maps of the Willis' effective impedance $\tilde{Z}^e_2 (k,\omega)$ (second-order LW-LF approximation) and its two-scale counterpart~$\tilde{\mathcal{Z}}^e_2(k,\omega)$ for the bi-laminate periodic structure~\eqref{1dmodel}.}\label{WillisTwoScale}
\end{figure}

\newpage 
\section{Summary and conclusions}

In this work we aim to expose the link between the Willis' effective description and the two-scale homogenization framework pertaining to the scalar wave motion in periodic media. To this end, we deploy the concept of effective impedance as a tool for comparison, and first formulate the Willis' model by the eigenfunction approach. The latter carries the advantage of (i)~seamlessly traversing the wavenumber-frequency space across dispersion curves, and (ii)~providing a clear insight into the phenomena of double points (e.g. intersecting dispersion curves) and the eigenmodes of zero mean that cannot be captured by the effective model. We next establish a long-wavelength, low-frequency (LW-LF) dispersive expansion of the Willis effective model, including terms up to the second order. Despite the intuitive expectation that such obtained effective impedance coincides with its two-scale counterpart, we find that the two descriptions differ by a modulation factor which is, up to the second order, expressible as a polynomial in frequency and wavenumber. We rigorously link this inconsistency to the fact that the two-scale homogenization is commonly restricted to the \emph{free-wave} solutions and thus fails to account for the body source term which, as it turns out, must also be homogenized -- by the reciprocal of the featured modulation factor. Through the exercise, we also discover that the operations of averaging (i.e. homogenization) and asymptotic expansion \emph{commute} when computing the second-order LW-LF approximation of the effective wave motion in periodic media. For generality, we further obtain the modulation factor for the two-scale homogenization of \emph{dipole} body sources, which may be relevant to some recent efforts to manipulate waves in metamaterials via e.g. a piezoelectric effect. The analysis presented herein, which amounts to a \emph{single-scale} expansion of the Willis' effective model, is inherently applicable to other asymptotic regimes such as the long-wavelength, finite-frequency (LW-FF) behavior which could be used to establish an effective description of the band gap(s) inside the first Brioullin zone. 

\section{Appendix}\label{App}

{\bf Proof of Lemma \ref{LemmaCellBasis}}:
From~(\ref{CellPDEfw})--(\ref{CellPDEfv}) and integration by parts, one has
\begin{eqnarray*}
-\int_Y \omega^2 \rho(\bx) \tilde{w}(\bx) \zeta_j(\bx)\, \textrm{d}\bx + \int_Y \big( G(\bx) \nabla_{\!\bk} \tilde{w}(\bx) \big) \cdot \overline{\nabla_{\!\bk} \zeta_m}(\bx)\, \textrm{d}\bx &=& \int_Y \zeta_m(\bx) dx, \\
-\int_Y \omega^2 \rho(\bx) \tilde{v}_j(\bx) \zeta_m(\bx)\, \textrm{d}\bx + \int_Y \big( G(\bx) (\nabla_{\!\bk} \tilde{v}_j(\bx) - \be_j )\big) \cdot \overline{\nabla_{\!\bk}\zeta _m}(\bx)\, \textrm{d}\bx &=& 0, 
\end{eqnarray*}
where $\zeta_m=\bzeta\!\cdot\!\be_m$, and~$\bzeta$ solves~\eqref{CellBasis}. Then 
\begin{eqnarray*}
&&-\int_Y \omega^2 \rho(\bx) \tilde{w}(\bx) \zeta_m(\bx)\, \textrm{d}\bx  + \int_Y \big( G(\bx) \nabla_{\!\bk} \tilde{w}(\bx) \big) \cdot \be_m \, \textrm{d}\bx ~=~ \\
&& \qquad \int_Y \zeta_m(\bx) \, \textrm{d}\bx - 
\int_Y \nabla_{\!\bk}\tilde{w}(\bx)\cdot \big( G(\bx)(\overline{\nabla_{\!\bk} \bzeta}_m(\bx) -\be_m) \big)\, \textrm{d}\bx ~=~ \int_Y \zeta_m(\bx) \, \textrm{d}\bx, \\
&&-\int_Y \omega^2 \rho(\bx) \tilde{v}_j(\bx) \zeta_m(\bx)\, \textrm{d}\bx + \int_Y \big( G(\bx) (\nabla_{\!\bk} \tilde{v}_j(\bx) - \be_j )\big) \cdot \be_m\, \textrm{d}\bx ~=~ \\
&& \qquad -\int_Y (\nabla_{\!\bk} \tilde{v}_j(\bx) - \be_j)\cdot \big( G(\bx)(\overline{\nabla_{\!\bk} \bzeta}_m(\bx) -\be_m) \big) \, \textrm{d}\bx ~=~ \int_Y  G(\bx) \be_j \big(\overline{\nabla_{\!\bk} \bzeta}_m(\bx)-\be_m\big) \, \textrm{d}\bx.
\end{eqnarray*}
Averaging above equations yields
\begin{eqnarray*}
\langle G \nabla_{\!\bk} \tilde{w} \rangle &\!\!\!=\!\!\!& \langle \overline{\bzeta} \rangle + \omega^2 \langle \rho \tilde{w} \overline{\bzeta} \rangle,  \\
\langle G \nabla_{\!\bk} \tilde{\bv} \rangle \cdot \be_m &\!\!\!=\!\!\!& \langle G \overline{\nabla_{\!\bk} \zeta_m} \rangle + \omega^2 \langle \rho \tilde{\bv} \zeta_m\rangle, 
\end{eqnarray*}
which proves the lemma. \proofend

\medskip

{\bf Proof of Lemma \ref{ReciprocityLemma}}: Thanks to Lemma~\ref{LemmaCellBasis}, we have 
\begin{eqnarray*}
\langle G \nabla_{\!\bk} \tilde{v}_j \rangle \cdot \be_m ~=~ \langle G \overline{\nabla_{\!\bk} \zeta_m} \rangle \cdot \be_j+ \omega^2 \langle \rho \tilde{v}_j \overline{\zeta}_m\rangle. 
\end{eqnarray*} 
On the other hand the multiplication of~(\ref{CellBasis}), written for~$\zeta_m$,  by $\overline{\zeta}_j$ and integration by parts yields
\begin{eqnarray*}
\langle G \overline{\nabla_{\!\bk} \zeta_j} \rangle \cdot \be_m  ~=~ \langle G \overline{\nabla_{\!\bk} \zeta_j} \nabla_{\!\bk} \zeta_m \rangle \qquad \Longrightarrow \qquad \langle G \overline{\nabla_{\!\bk} \zeta_m} \rangle \cdot \be_j ~=~ \overline{\langle G \overline{\nabla_{\!\bk} \zeta_j} \rangle \cdot \be_m}. 
\end{eqnarray*} 
From the eigenfunction expansion~(\ref{Eigenexpansionv}) of~$\tilde{\bv}$ one can obtain that for~$\bzeta$ by setting $\omega=0$, whereby 
\begin{eqnarray*}
\langle \rho \tilde{v}_j \overline{\zeta}_m\rangle ~=~ \sum_{n=1}^\infty  \frac{(G\nabla_{\!\bk}\tilde{\phi}_n,\be_m)}{  \tilde{\lambda}_n} \frac{ (\be_j, G \nabla_{\!\bk}\tilde{\phi}_n )  }{  \tilde{\lambda}_n - \omega^2  }.
\end{eqnarray*} 
This demonstrates that $\langle \rho \tilde{v}_j \overline{\zeta}_m\rangle = \overline{\langle \rho \tilde{v}_m \overline{\zeta}_j\rangle }$. As a result, $\langle G \nabla_{\!\bk} \tilde{v}_j \rangle \cdot \be_m = \overline{ \langle G \nabla_{\!\bk} \tilde{v}_m \rangle} \cdot \be_j$ which completes the proof. \proofend

\medskip

{\bf Proof of Lemma \ref{RelationsLemma}}: From~\eqref{CellPDEfv}, \eqref{CellPDEfvBC}, (\ref{Eigenexpansionv}), (\ref{Eigenexpansionw}), and~(\ref{CellBasis}) we have 
\begin{eqnarray*}
\tilde{w}~=~\sum_{n=1}^\infty  \frac{ (1, \tilde{\phi}_n) }{\tilde{\lambda}_n - \omega^2}\tilde{\phi}_n  , \quad
\tilde{\bv}~=~\sum_{n=1}^\infty  \frac{ (1, G \nabla_{\!\bk}\tilde{\phi}_n)}{\tilde{\lambda}_n - \omega^2} \tilde{\phi}_n, \quad
\tilde{\zeta}_m~=~ \sum_{n=1}^\infty  \frac{ (\be_m, G \nabla_{\!\bk}\tilde{\phi}_n)}{\tilde{\lambda}_n } \tilde{\phi}_n.
\end{eqnarray*}
On computing the $Y$-average of~(\ref{Eigensystem}), one finds
\begin{eqnarray*}
-i{\bk} \cdot \langle G \nabla_{\!\bk} \tilde{\phi}_{n} \rangle ~=~ \tilde{\lambda}_{n}  \langle \rho \tilde{\phi}_{n} \rangle.
\end{eqnarray*}
Since~$\tilde{\lambda}_n\in\mathbb{R}$, this further implies 
\begin{eqnarray*}
i\bk \!\cdot\! \langle \tilde{\bv} \rangle &\!\!\!\!=\!\!\!\!& 
\sum_{n=1}^\infty  \frac{ (i\bk, G \nabla_{\!\bk}\tilde{\phi}_n)}{\tilde{\lambda}_n - \omega^2} \langle\tilde{\phi}_n\rangle \:=\: \sum_{n=1}^\infty  \frac{ \tilde{\lambda}_{n}  (1, \rho \tilde{\phi}_{n} ) }{\tilde{\lambda}_n - \omega^2} \langle\tilde{\phi}_n\rangle \:=\: \sum_{n=1}^\infty   (1, \rho \tilde{\phi}_{n}) \langle\tilde{\phi}_n\rangle +\omega^2 \sum_{n=1}^\infty  \frac{  (1, \rho \tilde{\phi}_{n} ) }{\tilde{\lambda}_n - \omega^2} \langle\tilde{\phi}_n\rangle \\
&\!\!\!\!=\!\!& 1+\omega^2 \overline{\langle \rho \tilde{w} \rangle}, 
\end{eqnarray*}
which is the first claim of the lemma. Note that every infinite series in the above expression is convergent since $\tilde{w} \in L^2(Y)$ and $\tilde{v}_j \in L^2(Y)$. From Lemma~\ref{LemmaCellBasis}, we have 
\begin{eqnarray*}
\langle G \nabla_{\!\bk} \tilde{w} \rangle \cdot \be_m &\!\!\!\!=\!\!\!\!& \langle \overline{\zeta_m} \rangle + \omega^2 \langle \rho \tilde{w} \overline{\zeta}_m \rangle = \sum_{n=1}^\infty  \frac{\langle G \nabla_{\!\bk}\tilde{\phi}_n \rangle \cdot \be_m}{\tilde{\lambda}_n } \langle \overline{\tilde{\phi}_n} \rangle + 
\omega^2 \sum_{n=1}^\infty  \frac{ \langle G \nabla_{\!\bk}\tilde{\phi}_n \rangle \cdot \be_m}{\tilde{\lambda}_n (\tilde{\lambda}_n - \omega^2) } (1,\tilde{\phi}_n) \\
&\!\!\!\!=\!\!\!\!& \sum_{n=1}^\infty  \frac{ \langle G \nabla_{\!\bk}\tilde{\phi}_n \rangle \cdot \be_m}{  \tilde{\lambda}_n - \omega^2  } (1,\tilde{\phi}_n) ~=~ \overline{\langle \tilde{v}_m \rangle}, 
\end{eqnarray*}
which establishes the third claim. Thanks to Lemma~\ref{ReciprocityLemma}, one can show that 
\begin{eqnarray*}
\be_m \cdot \langle G \nabla_{\!\bk} (\tilde{\bv} \cdot i\bk) \rangle  &\!\!\!\!=\!\!\!\!&  \sum_j ik_j\langle G \nabla_{\!\bk} \tilde{v}_j \rangle \cdot \be_m ~=~ \sum_j ik_j \overline{\langle G \nabla_{\!\bk} \tilde{v}_m \rangle} \cdot \be_j ~=~ \overline{\langle G \nabla_{\!\bk} \tilde{v}_m \rangle} \cdot i\bk.
\end{eqnarray*}
On computing the $Y$-average of~(\ref{CellPDEfv}) written for~$\tilde{v}_m$, we find  
\begin{eqnarray*}
\overline{\langle G \nabla_{\!\bk} \tilde{v}_m \rangle} \cdot i\bk ~=~ \langle G \rangle \hh i \bk \cdot \be_m + \omega^2 \overline{\langle \rho \tilde{v}_m \rangle}, 
\end{eqnarray*}
so that~$\langle G \nabla_{\!\bk} (\tilde{\bv} \cdot i\bk) \rangle = i\bk \langle G\rangle + \omega^2 \overline{\langle \rho \tilde{\bv} \rangle}$. This proves the theorem.\proofend

\medskip

{\bf Proof of Lemma \ref{ComparisonSecondOrderHomogenizationeta0chi1}}: On multiplying~(\ref{Comparisoneta0}) by~$\bchi^{\mbox{\tiny{(1)}}}$ and integrating by parts, one obtains
\begin{eqnarray*}
-\int_Y (G\nabla\eta^{\mbox{\tiny{(0)}}}) \cdot\! \nabla \bchi^{\mbox{\tiny{(1)}}}\, \textrm{d}\bx  ~=~ 
\int_Y \frac{\rho\!-\!\rho_0}{\rho_0} \, \bchi^{\mbox{\tiny{(1)}}} \textrm{d}\bx.
\end{eqnarray*}
Since 
\begin{eqnarray*}
- \int_Y \nabla \eta^{\mbox{\tiny{(0)}}}\! \cdot G \big( \nabla \bchi^{\mbox{\tiny{(1)}}}\! + \boldsymbol{I} \big)\, \textrm{d}\bx ~=~
\int_Y  \eta^{\mbox{\tiny{(0)}}}\, \nabla \!\cdot\! \big(G(\nabla\bchi^{\mbox{\tiny{(1)}}} \!+\boldsymbol{I})\big)\, \textrm{d}\bx ~=~ \bzero 
\end{eqnarray*}
thanks to~\eqref{ComparisonEqnchi1}, it follows that 
\begin{eqnarray*}
\int_Y \big(G \nabla \eta^{\mbox{\tiny{(0)}}}  \big)\cdot  \boldsymbol{I}\, \textrm{d}\bx  ~=~ \int_Y \frac{\rho-\rho_0}{\rho_0} \bchi^{\mbox{\tiny{(1)}}} \textrm{d}\bx.
\end{eqnarray*}
On dividing the last equation by~$|Y|$ and recalling that $\langle\bchi^{\mbox{\tiny{(1)}}}\rangle=\bzero$, one completes the proof. \proofend

\medskip

{\bf Proof of Lemma \ref{LemmaEigenLeadingOrderSymmetry}}: Multiplying the $j$th component of~(\ref{ComparisonEqnchi1}) by $\chi_{\ell}^{\mbox{\tiny{(1)}}}=\bchi^{\mbox{\tiny{(1)}}}\!\cdot\be_\ell$ and integrating by parts gives 
\begin{eqnarray*} 
\int_Y \big( G (\nabla \chi^{\mbox{\tiny{(1)}}}_j +\boldsymbol{e}_j) \big) \cdot \nabla \chi^{\mbox{\tiny{(1)}}}_\ell\, \textrm{d}\bx ~=~ 0, 
\end{eqnarray*}
by which   
\begin{eqnarray*} 
\int_Y  G\hh \chi^{\mbox{\tiny{(1)}}}_{\ell,j}\, \textrm{d}\bx ~=\; - \int_Y  G\hh \nabla \chi^{\mbox{\tiny{(1)}}}_j \cdot \nabla \chi^{\mbox{\tiny{(1)}}}_\ell\, \textrm{d}\bx 
\qquad \Longrightarrow \qquad \langle G  \chi^{\mbox{\tiny{(1)}}}_{\ell,j}  \rangle ~=~ \langle G \chi^{\mbox{\tiny{(1)}}}_{j,\ell} \rangle 
\end{eqnarray*}
as claimed by the lemma. 
\proofend

\medskip

{\bf Proof of Lemma \ref{ComparisonSecondOrderHomogenizationAlpha1}}: Consider~(\ref{Comparisonalpha1}) written for~$\alpha^{\mbox{\tiny{(1)}}}_k=\balpha^{\mbox{\tiny{(1)}}}\!\cdot\be_k$. On multiplying this equation by $\chi^{\mbox{\tiny{(1)}}}_\ell$ and integrating by parts, one obtains 
\begin{eqnarray*}  
-\int_Y  \big( G \nabla \alpha^{\mbox{\tiny{(1)}}}_k \big) \cdot \nabla \chi^{\mbox{\tiny{(1)}}}_\ell \, \textrm{d}\bx ~=~ \int_Y \big(\rho \chi^{\mbox{\tiny{(1)}}}_k - \rho^{\mbox{\tiny{(1)}}}_k \big) \chi^{\mbox{\tiny{(1)}}}_\ell \, \textrm{d}\bx.
\end{eqnarray*}
Since 
\begin{eqnarray*}
- \int_Y \nabla  \alpha^{\mbox{\tiny{(1)}}}_k \cdot G \big( \nabla \chi^{\mbox{\tiny{(1)}}}_\ell + \boldsymbol{e}_\ell \big)\, \textrm{d}\bx ~=~
\int_Y   \alpha^{\mbox{\tiny{(1)}}}_k \: \nabla \!\cdot\! \Big( G \big( \nabla \chi^{\mbox{\tiny{(1)}}}_\ell + \boldsymbol{e}_\ell \big) \Big)\, \textrm{d}\bx ~=~ 0 
\end{eqnarray*}
thanks to~\eqref{ComparisonEqnchi1}, we have 
\begin{eqnarray*}  
\int_Y  \big( G \nabla \alpha^{\mbox{\tiny{(1)}}}_k \big)\cdot  \boldsymbol{e}_\ell \, \textrm{d}\bx ~=~ \int_Y \big(\rho \chi^{\mbox{\tiny{(1)}}}_k - \rho^{\mbox{\tiny{(1)}}}_k \big) \chi^{\mbox{\tiny{(1)}}}_\ell \, \textrm{d}\bx.
\end{eqnarray*}
Since~$\rho^{\mbox{\tiny{(1)}}}_k$ is a constant and~$\langle\chi_\ell^{\mbox{\tiny{(1)}}}\rangle=0$, the above equation yields 
$$
\langle G \alpha^{\mbox{\tiny{(1)}}}_{k,\ell} \rangle ~=~ \langle \rho \chi^{\mbox{\tiny{(1)}}}_k \chi^{\mbox{\tiny{(1)}}}_\ell \rangle 
$$
as claimed by the lemma.  \proofend



\end{document}